\documentclass[12pt, reqno]{amsart}

\usepackage{aligned-overset}
\usepackage{amsfonts}
\usepackage{amssymb}
\usepackage{latexsym}
\usepackage{graphics}
\usepackage[2emode]{psfrag}
\usepackage{mathrsfs}
\usepackage{amsthm}
\usepackage{amsmath}
\usepackage[all]{xy}
\usepackage{enumerate}
\usepackage{stmaryrd}
\usepackage{fancyhdr}
\usepackage[table]{xcolor}
\usepackage{cite}
\usepackage{color}
\definecolor{darkgreen}{rgb}{0,0.45,0}
\usepackage[pagebackref,colorlinks,citecolor=darkgreen,linkcolor=darkgreen]{hyperref}
\usepackage[capitalize]{cleveref}

\def\swvdash{\scalebox{0.7}{\rotatebox[origin=c]{135}{$\bot$}}}

\def\nevdash{\scalebox{0.7}{\rotatebox[origin=c]{225}{$\top$}}}
\def\nwvdashh{\scalebox{0.7}{\rotatebox[origin=c]{149}{$\top$}}}
\def\nwvdashhh{\scalebox{0.7}{\rotatebox[origin=c]{161}{$\top$}}}

\begin{document}

\renewcommand{\thefootnote}{\alph{footnote}}

\newcommand{\thlabel}[1]{\label{th:#1}}
\newcommand{\thref}[1]{Theorem~\ref{th:#1}}
\newcommand{\selabel}[1]{\label{se:#1}}
\newcommand{\seref}[1]{Section~\ref{se:#1}}
\newcommand{\lelabel}[1]{\label{le:#1}}
\newcommand{\leref}[1]{Lemma~\ref{le:#1}}
\newcommand{\prlabel}[1]{\label{pr:#1}}
\newcommand{\prref}[1]{Proposition~\ref{pr:#1}}
\newcommand{\colabel}[1]{\label{co:#1}}
\newcommand{\coref}[1]{Corollary~\ref{co:#1}}
\newcommand{\relabel}[1]{\label{re:#1}}
\newcommand{\reref}[1]{Remark~\ref{re:#1}}
\newcommand{\exlabel}[1]{\label{ex:#1}}
\newcommand{\exref}[1]{Example~\ref{ex:#1}}
\newcommand{\delabel}[1]{\label{de:#1}}
\newcommand{\deref}[1]{Definition~\ref{de:#1}}
\newcommand{\eqlabel}[1]{\label{eq:#1}}
\newcommand{\equref}[1]{(\ref{eq:#1})}

\newcommand{\Hom}{{\sf Hom}}
\newcommand{\End}{{\sf End}}
\newcommand{\Ext}{{\sf Ext}}
\newcommand{\Fun}{{\sf Fun}}
\newcommand{\Mor}{{\sf Mor}\,}
\newcommand{\Aut}{{\sf Aut}\,}
\newcommand{\Ann}{{\sf Ann}\,}
\newcommand{\Ker}{{\sf Ker}\,}
\newcommand{\Coker}{{\sf Coker}\,}
\newcommand{\im}{{\sf Im}\,}
\newcommand{\coim}{{\sf Coim}\,}
\newcommand{\Trace}{{\sf Trace}\,}
\newcommand{\Char}{{\sf Char}\,}
\newcommand{\Mod}{{\sf Mod}}
\newcommand{\Vect}{{\sf Vect}}
\newcommand{\VectGr}{{\sf Vect-Grph}}
\newcommand{\Alg}{{\sf Alg}}
\newcommand{\Coalg}{{\sf Coalg}}
\newcommand{\Bialg}{{\sf Bialg}}
\newcommand{\Hopf}{{\sf Hopf}}
\newcommand{\sHopf}{{\sf sHopf}}
\newcommand{\Comon}{{\sf Comon}}
\newcommand{\Mon}{{\sf Mon}}
\newcommand{\Spec}{{\sf Spec}\,}
\newcommand{\Span}{{\sf Span}\,}
\newcommand{\sgn}{{\sf sgn}\,}
\newcommand{\Id}{{\sf Id}\,}
\newcommand{\Com}{{\sf Com}\,}
\newcommand{\codim}{{\sf codim}}
\newcommand{\Mat}{{\sf Mat}}
\newcommand{\Coint}{{\rm Coint}}
\newcommand{\Incoint}{{\sf Incoint}}
\newcommand{\can}{{\sf can}}
\newcommand{\Bim}{{\sf Bim}}
\newcommand{\CAT}{{\sf CAT}}
\newcommand{\Cat}{{\sf Cat}}
\newcommand{\LMonCat}{{\sf LMonCat}}
\newcommand{\opCat}{{\sf opCat}}
\newcommand{\Grph}{{\sf Grph}}
\newcommand{\Fam}{{\sf Fam}}
\newcommand{\Maf}{{\sf Maf}}
\newcommand{\VCat}{{\textsf{$\Vv${\sf-Cat}}}}
\newcommand{\VHopf}{{\textsf{$\Vv${\sf-Hopf}}}}
\newcommand{\VsHopf}{{\textsf{$\Vv${\sf-sHopf}}}}
\newcommand{\WCat}{{\textsf{$\Ww${\sf-Cat}}}}
\newcommand{\FCat}{{\textsf{$F${\sf-Cat}}}}
\newcommand{\GCat}{{\textsf{$G${\sf-Cat}}}}
\newcommand{\VopCat}{{\textsf{$\Vv${\sf-opCat}}}}
\newcommand{\VGrph}{{\textsf{$\Vv${\sf-Grph}}}}
\newcommand{\WGrph}{{\textsf{$\Ww${\sf-Grph}}}}
\newcommand{\FGrph}{{\textsf{$F${\sf-Grph}}}}
\newcommand{\GGrph}{{\textsf{$G${\sf-Grph}}}}
\newcommand{\aGrph}{{\textsf{$\alpha${\sf-Grph}}}}
\newcommand{\VopGrph}{{\textsf{$\Vv${\sf-opGrph}}}}
\newcommand{\sign}{{\sf sign}}
\newcommand{\kar}{{\sf kar}}
\newcommand{\rad}{{\sf rad}}
\newcommand{\Rat}{{\sf Rat}}
\newcommand{\Cob}{{\sf Cob}}
\newcommand{\ev}{{\sf ev}}
\newcommand{\sd}{{\sf d}}
\def\colim{{\sf colim}\,}
\def\lim{{\sf lim}\,}
\def\tildej{\tilde{\jmath}}
\def\barj{\bar{\jmath}}

\def\Ab{\underline{\underline{\sf Ab}}}
\def\lan{\langle}
\def\ran{\rangle}
\def\ot{\otimes}
\def\bul{\bullet}
\def\ubul{\underline{\bullet}}

\def\id{\textrm{{\small 1}\normalsize\!\!1}}
\def\To{{\multimap\!\to}}
\def\bigperp{{\LARGE\textrm{$\perp$}}} 
\newcommand{\QED}{\hspace{\stretch{1}}
\makebox[0mm][r]{$\Box$}\\}

\def\AA{{\mathbb A}}
\def\BB{{\mathbb B}}
\def\CC{{\mathbb C}}
\def\DD{{\mathbb D}}
\def\EE{{\mathbb E}}
\def\FF{{\mathbb F}}
\def\GG{{\mathbb G}}
\def\HH{{\mathbb H}}
\def\II{{\mathbb I}}
\def\JJ{{\mathbb J}}
\def\KK{{\mathbb K}}
\def\LL{{\mathbb L}}
\def\MM{{\mathbb M}}
\def\NN{{\mathbb N}}
\def\OO{{\mathbb O}}
\def\PP{{\mathbb P}}
\def\QQ{{\mathbb Q}}
\def\RR{{\mathbb R}}
\def\TT{{\mathbb T}}
\def\UU{{\mathbb U}}
\def\VV{{\mathbb V}}
\def\WW{{\mathbb W}}
\def\XX{{\mathbb X}}
\def\YY{{\mathbb Y}}
\def\ZZ{{\mathbb Z}}

\def\aa{{\mathfrak A}}
\def\bb{{\mathfrak B}}
\def\cc{{\mathfrak C}}
\def\dd{{\mathfrak D}}
\def\ee{{\mathfrak E}}
\def\ff{{\mathfrak F}}
\def\gg{{\mathfrak G}}
\def\hh{{\mathfrak H}}
\def\ii{{\mathfrak I}}
\def\jj{{\mathfrak J}}
\def\kk{{\mathfrak K}}
\def\ll{{\mathfrak L}}
\def\mm{{\mathfrak M}}
\def\nn{{\mathfrak N}}
\def\oo{{\mathfrak O}}
\def\pp{{\mathfrak P}}
\def\qq{{\mathfrak Q}}
\def\rr{{\mathfrak R}}
\def\tt{{\mathfrak T}}
\def\uu{{\mathfrak U}}
\def\vv{{\mathfrak V}}
\def\ww{{\mathfrak W}}
\def\xx{{\mathfrak X}}
\def\yy{{\mathfrak Y}}
\def\zz{{\mathfrak Z}}

\def\aaa{{\mathfrak a}}
\def\bbb{{\mathfrak b}}
\def\ccc{{\mathfrak c}}
\def\ddd{{\mathfrak d}}
\def\eee{{\mathfrak e}}
\def\fff{{\mathfrak f}}
\def\ggg{{\mathfrak g}}
\def\hhh{{\mathfrak h}}
\def\iii{{\mathfrak i}}
\def\jjj{{\mathfrak j}}
\def\kkk{{\mathfrak k}}
\def\lll{{\mathfrak l}}
\def\mmm{{\mathfrak m}}
\def\nnn{{\mathfrak n}}
\def\ooo{{\mathfrak o}}
\def\ppp{{\mathfrak p}}
\def\qqq{{\mathfrak q}}
\def\rrr{{\mathfrak r}}
\def\sss{{\mathfrak s}}
\def\ttt{{\mathfrak t}}
\def\uuu{{\mathfrak u}}
\def\vvv{{\mathfrak v}}
\def\www{{\mathfrak w}}
\def\xxx{{\mathfrak x}}
\def\yyy{{\mathfrak y}}
\def\zzz{{\mathfrak z}}

\newcommand{\aA}{\mathscr{A}}
\newcommand{\bB}{\mathscr{B}}
\newcommand{\cC}{\mathscr{C}}
\newcommand{\dD}{\mathscr{D}}
\newcommand{\eE}{\mathscr{E}}
\newcommand{\fF}{\mathscr{F}}
\newcommand{\gG}{\mathscr{G}}
\newcommand{\hH}{\mathscr{H}}
\newcommand{\iI}{\mathscr{I}}
\newcommand{\jJ}{\mathscr{J}}
\newcommand{\kK}{\mathscr{K}}
\newcommand{\lL}{\mathscr{L}}
\newcommand{\mM}{\mathscr{M}}
\newcommand{\nN}{\mathscr{N}}
\newcommand{\oO}{\mathscr{O}}
\newcommand{\pP}{\mathscr{P}}
\newcommand{\qQ}{\mathscr{Q}}
\newcommand{\rR}{\mathscr{R}}
\newcommand{\sS}{\mathscr{S}}
\newcommand{\tT}{\mathscr{T}}
\newcommand{\uU}{\mathscr{U}}
\newcommand{\vV}{\mathscr{V}}
\newcommand{\wW}{\mathscr{W}}
\newcommand{\xX}{\mathscr{X}}
\newcommand{\yY}{\mathscr{Y}}
\newcommand{\zZ}{\mathscr{Z}}

\newcommand{\Aa}{\mathcal{A}}
\newcommand{\Bb}{\mathcal{B}}
\newcommand{\Cc}{\mathcal{C}}
\newcommand{\Dd}{\mathcal{D}}
\newcommand{\Ee}{\mathcal{E}}
\newcommand{\Ff}{\mathcal{F}}
\newcommand{\Gg}{\mathcal{G}}
\newcommand{\Hh}{\mathcal{H}}
\newcommand{\Ii}{\mathcal{I}}
\newcommand{\Jj}{\mathcal{J}}
\newcommand{\Kk}{\mathcal{K}}
\newcommand{\Ll}{\mathcal{L}}
\newcommand{\Mm}{\mathcal{M}}
\newcommand{\Nn}{\mathcal{N}}
\newcommand{\Oo}{\mathcal{O}}
\newcommand{\Pp}{\mathcal{P}}
\newcommand{\Qq}{\mathcal{Q}}
\newcommand{\Rr}{\mathcal{R}}
\newcommand{\Ss}{\mathcal{S}}
\newcommand{\Tt}{\mathcal{T}}
\newcommand{\Uu}{\mathcal{U}}
\newcommand{\Vv}{\mathcal{V}}
\newcommand{\Ww}{\mathcal{W}}
\newcommand{\Xx}{\mathcal{X}}
\newcommand{\Yy}{\mathcal{Y}}
\newcommand{\Zz}{\mathcal{Z}}

\def\units{{\mathbb G}_m}
\def\rightact{\hbox{$\leftharpoonup$}}
\def\leftact{\hbox{$\rightharpoonup$}}

\def\*C{{}^*\hspace*{-1pt}{\Cc}}
\def\*c{{}^*\hspace*{-1pt}{\cc}}

\def\text#1{{\rm {\rm #1}}}

\def\smashco{\mathrel>\joinrel\mathrel\triangleleft}
\def\cosmash{\mathrel\triangleright\joinrel\mathrel<}

\def\ol{\overline}
\def\ul{\underline}
\def\dul#1{\underline{\underline{#1}}}
\def\Nat{{\sf Nat}}
\def\Set{{\sf Set}}
\def\MCl{\dul{\rm MCl}}

\renewcommand{\subjclassname}{\textup{2000} Mathematics Subject
     Classification}

\newtheorem{proposition}{Proposition}[section] 
\newtheorem{lemma}[proposition]{Lemma}
\newtheorem{corollary}[proposition]{Corollary}
\newtheorem{theorem}[proposition]{Theorem}

\theoremstyle{definition}
\newtheorem{Definition}[proposition]{Definition}
\newtheorem{definition}[proposition]{Definition}
\newtheorem{example}[proposition]{Example}
\newtheorem{examples}[proposition]{Examples}

\theoremstyle{remark}
\newtheorem{remarks}[proposition]{Remarks}
\newtheorem{remark}[proposition]{Remark}

\title{Free and co-free constructions for Hopf categories}

\author[P. Gro{\ss}kopf]{Paul Gro{\ss}kopf}
\address{Paul Gro{\ss}kopf, D\'epartement de Math\'ematiques, Universit\'e Libre de Bruxelles, Belgium}
\email{paul.grosskopf@gmx.at}

\author[J. Vercruysse]{Joost Vercruysse}
\address{Joost Vercruysse, D\'epartement de Math\'ematiques, Universit\'e Libre de Bruxelles, Belgium}
\email{joost.vercruysse@ulb.be}

\begin{abstract}
We show that under mild conditions on the monoidal base category $\Vv$, the category $\VHopf$ of Hopf $\Vv$-categories is locally presentable and deduce the existence of free and cofree Hopf categories. We also provide an explicit description of the free and cofree Hopf categories over a semi-Hopf category. One of the conditions on the base category $\Vv$, states that endofunctors obtained by tensoring with a fixed object preserve jointly monic families, which leads us to the notion of ``very flat monoidal product'', which we investigate in particular for module categories.
\end{abstract}

\maketitle

\tableofcontents

\section*{Introduction}

Free constructions are a fundamental tool for many algebraic structures. In particular, the free group over a set or a monoid is a classical construction that allows to describe all possible symmetries of a given object, in terms of generators and relations. Formally, free constructions provide left adjoints to forgetful functors. The free group over a set can be recovered combining the construction of a free monoid over a set, and the free group over a monoid. 
\[
\xymatrix{
{\sf Grp} \ar[rr]_\bot^\bot && {\sf Mon} \ar@<-1ex>[rr]^\bot \ar@<-2ex>[ll] \ar@<2ex>[ll] && \Set \ar@<-1ex>[ll]
}
\]
In fact, the forgetful functor from groups to monoids also has a right adjoint, which is obtained by taking all invertible elements in a given monoid. One can think of this construction as the ``cofree group over a monoid''. 

Moving from classical to non-commutative geometry, symmetries are no longer described in terms of groups and their actions, but rather by Hopf algebras (or quantum groups) and their (co)actions (see e.g.\ \cite{Manin}), which leads also to ``hidden symmetries'' of classical geometric objects. 
Replacing groups by Hopf algebras, means that we substitute the Cartesian category of sets, by the monoidal category of vector spaces.
A main consequence of this is that, in contrast to the Cartesian case of sets where every object has a unique coalgebra structure, coalgebra structures in the category of vector spaces can be more involved and the above picture needs to be completed with additional forgetful functors from bialgebras to algebras and from coalgebras to vector spaces. In contrast to the classical forgetful functors, these functors usually do no longer have a left adjoint, but rather a right adjoint, providing the cofree coalgebra over a vector space, and the cofree bialgebra over an algebra, see \cite[Section 6.4]{Sweedler}. As in the case of groups above, the forgetful functor from Hopf algebras to bialgebras has both a left and right adjoint. The existence of these adjoints was suggested in \cite{Sweedler}, although without providing a proof. The left adjoint, that is the free Hopf algebra or Hopf envelope of a bialgebra, was first constructed by Manin \cite{Manin} and generalizes the construction of the free Hopf algebra over a coalgebra given earlier by Takeuchi \cite{Takeuchi}, as the latter arises by combining Manin's construction with the free bialgebra over a coalgebra (which is just given by the tensor algebra over the underlying vector space of the coalgebra, see e.g.\ \cite{Sweedler}). For the right adjoint, a first proof of the existence of the cofree Hopf algebra over a bialgebra was given in \cite{Ago1}, based on earlier results from \cite{Porst:JPAA}. All these free and cofree constructions are summarized in the diagram below. 
\[
\xymatrix@!C{
&&& {\sf Alg} \ar@<1ex>[dl]^{B^c} \ar@<1ex>[dr] \ar@{}[dr]|-{\swvdash} \\
{\sf Hopf} \ar[rr]_\bot^\bot && {\sf Bialg} \ar@<1ex>[ur] \ar@<1ex>[dr] \ar@{}[ur]|-{\nevdash} \ar@{}[dr]|-{\swvdash} 
\ar@<-2ex>[ll] \ar@<2ex>[ll] && \Vect \ar@<1ex>[ul]^-{T}  \ar@<1ex>[dl]^-{T^c}\\
&&& {\sf Coalg} \ar@<1ex>[ul]^-{B} \ar@<1ex>[ur] \ar@{}[ur]|-{\nevdash}
}
\]
The square of this diagram commutes in the following sense. Of course, all combinations of forgetful functors from bialgebras to vector spaces coincide. Also, the composition of functors between categories of algebras and coalgebras coincide. This expresses, for example, the fact that the underlying algebra of the free bialgebra over a coalgebra coincides with the free algebra over the underlying vector space of this coalgebra. It might be useful to remark that this commutativity rule in combination with the adjunctions, leads to a natural transformation from the free bialgebra over the cofree coalgebra over a vector space $V$ to the cofree bialgebra over the free algebra over $V$ (we abusively denote all forgetful functors by $U$):
\begin{eqnarray*}
\Hom_{\Bialg}(B T^c V, B^c T V)&\cong& \Hom_{\Coalg}(T^c V, U B^c T V)\\
\cong \Hom_{\Coalg}(T^c V, T^c U T V) &\cong& \Hom_{\Vect}(UT^c V, UT V) \ni \eta_V \circ \epsilon_V,
\end{eqnarray*}
where $\eta$ and $\epsilon$ denote the unit and counit of respectively the adjunction $(T,U)$ and $(U,T^c)$.
 However, this is never an isomorphism, as one can see already in case $V$ is the zero space, since $B T^c 0 = B k = k[x]$ is the polynomial (bi)algebra and $B^c T 0 = B^c k$ is the (strictly bigger) bialgebra of linearly recursive sequences.

Closely related to free and cofree constructions, is the question of completeness and cocompleteness of the various categories involved. For the category of Hopf algebras, this was again already claimed in \cite{Sweedler}, but a full proof was given only relatively recent, based on the theory of locally presentable categories \cite{Porst}, or by means of more explicit constructions of limits and colimits \cite{Ago1, Ago2}. These works culminated in a systematic study of the category of Hopf algebras, not only over vector spaces but over general symmetric monoidal categories in \cite{Porst:formal1, Porst:formal2}.

In this paper, we will investigate a ``multi-object'' version of the above (see \eqref{diagramconclusion} at the end of this paper for a summarizing diagram similar to the one above), motivated by the fact that in many (geometric) situations, the natural objects that arise are groupoids rather than groups.
With the rise of non-commutative geometry, several algebraic structures have been introduced to serve as a non-commutative or linear counterpart for groupoids. For example, weak Hopf algebras \cite{BNS:WHA} and Hopf algebroids \cite{Bohm:Hoids} both can play this role.
More recently, Hopf categories \cite{BCV} have been introduced as an alternative approach. A semi-Hopf $\Vv$-category over a symmetric monoidal category $\Vv$, is a category enriched over the category of $\Vv$-coalgebras. If a semi-Hopf $\Vv$-category moreover admits a suitable notion of an antipode it is called a Hopf $\Vv$-category. When $\Vv$ is the category of sets, then we recover the usual notion of a groupoid. On the other hand, a Hopf $\Vv$-category with one object is exactly a Hopf algebra in $\Vv$.  In this sense, the theory of Hopf $\Vv$-categories naturally unifies the theory of Hopf algebras with the theory of groupoids. Moreover, if a Hopf category has only a finite set of objects, it was shown in \cite[Section 7]{BCV} that the coproduct of all Hom-objects is in a natural way a weak Hopf algebra, and hence a Hopf algebroid (over a commutative base algebra). However, Hopf categories with an arbitrary set, and even a class of objects, can also be considered. Such objects can no longer be described in terms of a weak Hopf algebra or a Hopf algebroid, which is exactly one of the main advantages of working with Hopf $\Vv$-categories rather than weak Hopf algebras or Hopf algebroids. On the other hand, the theory of Hopf categories is in a sense easier to handle and formally much closer to classical Hopf algebras than weak Hopf algebras or Hopf algebroids.
Hopf categories were also shown to fit in the more general theory of Hopf monads in monoidal bicategories \cite{Bohm:polyads} and alternatively can be described as oplax Hopf monoids in a braided monoidal bicategory \cite{BFVV2}. Galois and descent theory for Hopf categories has been developed in \cite{CaeFie} and a version of the celebrated Larson-Sweedler theorem has been proven in \cite{BFVV}. 

Applying the machinery of locally presentable categories and existing results from literature, we derive that the category of semi-Hopf $\Vv$-categories is locally presentable if $\Vv$ is and some mild conditions on $\Vv$ are fulfilled, see \cref{sHopflp}. 
The case of Hopf $\Vv$-categories shows to be more complicated, and in order to prove that the category of Hopf $\Vv$-categories is also locally presentable, we will make use of an explicit construction of limits and colimits in the category of semi-Hopf categories. Especially the case of limits needs particular attention, as these make use of limits in the category of coalgebras. In order to describe the latter explicitly, one needs certain exactness properties for the monoidal product. This leads us to the notion of a ``very flat'' monoidal product, which requires endofunctors obtained from tensoring with a fixed object to preserve arbitrary jointly monic families. We study such monoidal products from a module theoretic point of view in \seref{veryflat}, and we believe this might be of independent interest.

In \seref{cofree} we provide an explicit construction of the cofree coalgebra over an object in a locally presentable monoidal category with very flat monoidal product. Although the existence of such objects was already known (see e.g. \cite[2.7]{Porst}), an explicit construction only had been given in case of vector spaces over a field, which is vastly generalized here. Moreover, for a limit of coalgebras in a monoidal category $\Vv$ with very flat monoidal product, our construction leads to an explicit family of jointly monic morphisms with this limit as domain, in the underlying monoidal category $\Vv$ (see \cref{limitsincoalg}). These families allow us to make explicit computations for such limits.

Based on the above, we can then deduce (see \seref{main}) that the category of Hopf $\Vv$-categories is again locally presentable and deduce the existence of free and cofree Hopf $\Vv$-categories, generalizing the picture of the ``one-object-case'' above. It is important to remark that these free and cofree constructions leave the set (or class) of objects in the considered $\Vv$-category unchanged. We end the paper by providing the explicit description of the free and cofree Hopf categories over a given semi-Hopf category in \seref{construction}. One can observe that our construction naturally unifies the description of free Hopf algebras (see \cite{Manin}) with the one of free groupoids (see e.g.\ \cite{Brown}). 

A particularly motivating example of a semi-Hopf category is the category of algebras, which is known to be enriched over coalgebras by means of Sweedler's universal measuring coalgebras, see \cite{Sweedler}, \cite{Tambara}, \cite{Chris}, \cite{AGV}. The results of our paper allow to consider the free or cofree Hopf category over this semi-Hopf category, which then leads then to a Hopf category structure on the category of algebras. Following \cite{Manin}, one could consider the Hom-objects of this Hopf category as describing the natural non-commutative (iso)morphisms between the non-commutative spaces whose coordinate algebras are the objects of this category.

\section{Preliminaries on $\Vv$-graphs and $\Vv$-categories}

{
In this section we recall several notions and results that are scattered throughout literature and that will be useful later in this paper. We don't claim any originality for the results presented here, although some observations might be new or formulated differently from existing literature. We tried to provide many pointers to literature, without any claim of being complete.
}

\subsection{Locally presentable categories}

Recall (see e.g.\ \cite{AdaRos}) that a category $\Vv$ is called {\em locally presentable} if and only if $\Vv$ is cocomplete and there exists a (small) set of objects $\Ss$ and a regular cardinal $\lambda$ with the following two properties:
\begin{enumerate}[(i)]
\item the representable functor $\Hom(S,-):\Vv\to\Set$ preserves $\lambda$-filtered colimits for any $S\in \Ss$;  
\item every object in $\Vv$ is a $\lambda$-filtered colimit of elements in $S$.
\end{enumerate}
Locally presentable categories are known to satisfy a lot of interesting properties: they have a generator, are complete and cocomplete, well-powered and co-well-powered. In general they do not have a cogenerator. The dual of a locally presentable category $\Vv$ is not locally presentable, unless $\Vv$ is equivalent to a complete lattice.
Among the many examples of locally presentable categories, let us just mention especially $\Set$, the category of sets, and $\Mod_k$, the category of modules over a (commutative) ring $k$, which will be our main cases of interest in what follows. However, the theory has much wider applications since many more examples exist, such as the category of Banach spaces \cite{Ros} and the category of $C^*$-algebras, to mention but a few examples.

The reason why we consider locally presentable categories in this paper, is that they offer the following powerful adjoint functor theorem.

\begin{proposition}\label{adjointlp}
Let $F:\Vv\to \Ww$ be a functor between locally presentable categories. 
\begin{enumerate}
\item $F$ has a right adjoint if and only if it preserves colimits;
\item $F$ has a left adjoint if and only if it preserves limits and $\lambda$-filtered colimits for some regular cardinal $\lambda$.
\end{enumerate}
\end{proposition}

\begin{proof}
\ul{(i)} follows from the Special Adjoint Functor Theorem in combination with the observations made above that a locally presentable category is cocomplete, co-well-powered and has a generator. \ul{(ii)} is proven in \cite[Theorem 1.66]{AdaRos}.
\end{proof}

We will also use the following results.

\begin{proposition}[{\cite[Corollary 2.48]{AdaRos}}]\label{sublp}
A full subcategory of a locally presentable category, which is closed under limits and colimits, is locally presentable.
\end{proposition}

\begin{proposition}[{\cite[Proposition 1.61]{AdaRos}}]\label{epi-mono}
Every locally presentable category has both (strong epi, mono)- and (epi, strong mono)-factorizations of morphisms.
\end{proposition}

\subsection{The category of $\Vv$-graphs} 

\begin{Definition}
Let $\Vv$ be an arbitrary category. A {\em $\Vv$-graph} is a pair $\ul A=(A^0,A^1)$ where $A^0$ is a set\footnote{In fact, the definition also makes sense when $A^0$ is a class, but to avoid set-theoretical issues we will mainly consider the case where $A^0$ is a set.} and $A^1$ is a family of objects $A^1=(A_{x,y})_{x,y\in A^0}$ in $\Vv$ indexed by $A^0\times A^0$. The elements of $A^0$ are called the `objects' of $\ul A$, the $\Vv$-objects $A_{x,y}$ are the called the `Hom-objects' of $\ul A$.
By sheer laziness we also write $A_{x,y}=A_{xy}$. If $\ul A=(A^0,A^1)$ is a $\Vv$-graph then we also say that ``$\ul A$ is a $\Vv$-graph over $A^0$''.\\
Let $\ul A$ and $\ul B$ be two $\Vv$-graphs. We define a {\em morphism of $\Vv$-graphs} $\ul f:\ul A\to \ul B$ as a pair $\ul f=(f^0,f^1)$ where $f^0:A^0\to B^0$ is a map and $f^1=\left (f_{x,y}: A_{x,y} \to B_{fx,fy} \right )_{x,y \in A^0}$ is a family of morphisms in $\Vv$. Remark that we write $f^0(x)=fx$ for $x\in A^0$ for sake of brevity.

The composition of two morphisms of $\Vv$-graphs $\ul f:\ul A\to \ul B$ and $\ul g: \ul B\to \ul C$ is given by 
$(g\circ f)^0=g^0\circ f^0$ and $(g\circ f)^1=\left(g_{fx,fy} \circ f_{x,y}: A_{x,y} \to C_{gfx,gfy} \right)_{x,y \in A^0}$
This defines the category of $\Vv$-graphs that we denote as $\VGrph$.
\end{Definition}

\begin{remarks}
\begin{enumerate}
\item One can also define the category $\VopGrph$ as the category $(\textsf{$\Vv^{op}$-Grph})^{op}$. Explicitly, this category has the same objects as $\VGrph$ but a morphism $\ul f:\ul A\to \ul B$ is a pair $\ul f=(f^0,f^1)$, where $f^0:B^0\to A^0$ is a map and $f^1$ is a family of morphisms $\left( f_{x,y}:A_{fx,fy}\to B_{x,y}\right)_{x,y\in B^0}$.
\item
Let us remark that the category $\VGrph$ can be viewed as a full subcategory of the category $\Fam(\Vv)$ of ``families'' over $\Vv$, which is the free completion of $\Vv$ under coproducts (see e.g.\ \cite[Section 4.2]{GV}). In fact, $\VGrph$ can also be interpreted as the free completion of $\Vv$ under ``double coproducts'', by which we mean coproducts indexed over sets of type $X\times X$, where $X$ is a set.
Similarly, the category $\VopGrph$ is a full subcategory of the category $\Maf(\Vv)$, which is the free completion under products. Recently, in \cite{AGM} the duality between and pre-rigidity of the categories $\Fam(\Vv)$ and $\Maf(\Vv)$ has been investigated. Probably several of the results developed there could be restricted to the categories $\VGrph$ and $\VopGrph$.
\end{enumerate}
\end{remarks}

\begin{definition}
If $\ul A$ is a graph, then its {\em opposite graph} is the graph $\ul A^{op}$ with the same set of objects $A^0$ and for any $x,y\in A^0$, we have $A^{op}_{x,y}=A_{y,x}$. 
\end{definition}

With the above definition, one observes the following.

\begin{proposition}\label{opgrph}
There is an involutive endofunctor $(-)^{op}:\VGrph\to\VGrph$, which sends a $\Vv$-graph to its opposite $\Vv$-graph.
\end{proposition}

The following observation about ``base change'', will be useful later on.

\begin{proposition}\label{Grphfunctor}
Any functor $F:\Vv\to\Ww$ induces a functor $\FGrph:\VGrph\to\WGrph$. Similarly, if $G:\Vv\to\Ww$ is another functor, then any natural transformation $\alpha:F\Rightarrow G$ induces a natural transformation $\aGrph:\FGrph\Rightarrow\GGrph$. In fact, we obtain a $2$-functor $\Grph:\Cat\to \Cat$. It follows that any adjunction between $\Vv$ and $\Ww$ also induces an adjunction between $\VGrph$ and $\WGrph$.
\end{proposition}

\begin{proof}
The $2$-functor  $\Grph$ is defined entirely component-wise. More explicitly, for any $\Vv$-graph $\ul A=(A^0,A^1)$, we define $\FGrph(\ul A)^0=A^0$ and $\FGrph(\ul A)^1=(F(A_{xy}))_{x,y\in A^0}$. Similarly, for any morphism of $\Vv$-graphs $\ul f=(f^0,f^1):\ul A\to \ul B$, we define $\FGrph(\ul f)^0=f^0$ and $\FGrph(\ul f)^1=(F(f_{xy}))_{x,y\in A^0}$. Finally, for any object $\ul A$, let us define $\aGrph_{\ul A}=(\alpha^0_A,\alpha^1_A)$ by $\alpha_A^0=id:A^0\to A^0$ and $(\alpha_A)_{xy}=\alpha_{A_{xy}}:F(A_{xy}) \to G(A_{xy})$.

We leave the verification that this construction provides indeed a $2$-functor to reader.
\end{proof}

{ In the next few propositions are build up with the aim to show that certain constructions in $\VGrph$ (used in further when dealing with limits and colimits in more involved categories) can be reduced to constructions in the underlying categories $\Set$ and $\Vv$.}

\begin{proposition}\label{Pff}
Consider the functor $P:\Vv\to\VGrph$, which sends an object $V$ in $\Vv$ to the $\Vv$-graph $(\{*\},V)$ over the (fixed) singleton set $\{*\}$. Then we have the following.
\begin{enumerate}[(i)]
\item $P$ is a full and faithful functor.
\item The functor $P$ has a left adjoint if and only if $\Vv$ has double coproducts (i.e. coproducts indexed by sets of the form $X\times X$ where $X$ is a set).
\end{enumerate}
\end{proposition}

\begin{proof}
\ul{(i)}.
Consider objects $V$ and $W$ in $\Vv$.
A morphism $\ul f:P(V)\to P(W)$ is then defined as a pair $(f^0,f^1)$, where $f^0:\{*\}\to \{*\}$ has to be the identity map, and $f^1$ consists of one single morphism $f:V\to W$. This observation shows that $P$ is fully faithful.

\ul{(ii)}. 
This follows from $\VGrph$ being the free completion of $\Vv$ under double coproducts. Explicitly, if we denote the left adjoint of $P$ by $L$, then for any object $\ul A$, the object $L{\ul A}\in\Vv$ must satisfy the property
$$\VGrph(\ul A,P(V))\cong \Vv(L{\ul A},V)$$
Moreover, for any $\ul f=(f^0,f^1)\in \VGrph(\ul A,P(V))$, we have that $f^0:A^0\to \{*\}$ is the unique map sending every element of $A^0$ to $*$. Henceforth, $f^1$ is just a collection of $\Vv$-morphisms $f_{x,y}:A_{x,y}\to V$. The natural bijection above tells that for any such collection, there exists a unique morphism $L({\ul A})\to V$, which expresses exactly the universal property of $L(\ul A)$ as the coproduct of the family of objects $A_{x,y}$ in $\Vv$. Remark that the canonical maps $A_{x,y}\to \coprod_{x,y} A_{x,y}=L(\ul A)$ correspond exactly to the unit of the adjunction.
\end{proof}

\begin{proposition}[See also \cite{KelLac}]\label{Upreserves}
Consider the forgetful functor $U:\VGrph\to\Set$.
\begin{enumerate}[(i)]
\item $\Vv$ has an initial object if and only if $U$ has a fully faithful left adjoint.
\item $\Vv$ has a terminal object if and only if $U$ has a fully faithful right adjoint.
\end{enumerate}
In particular, if $\Vv$ has a terminal and initial object, then $U$ preserves all limits and colimits.
\end{proposition}

\begin{proof}
In case $\Vv$ has an initial object $\bot$, then the forgetful functor $U$ has a left adjoint $L$ which assigns to any set $X$ the graph over $X$ which has the initial object in every component, i.e. $L(X)=(X^0,X^1)$, where $X^0=X$ and $X_{x,y}=\bot$ for all $x,y\in X$. By construction, we have that $X=UL(X)$, which is exactly the unit of the adjunction, so $L$ is fully faithful.

Conversely, suppose that $U$ has a fully faithful left adjoint $L$, then we know that the underlying set of $L(X)$ is (up to bijection) just $X$, for any set $X$. Denote $L(\{*\})=(\{*\},\bot)$.
Consider any object $V$ in $\Vv$ and associate to it the $\Vv$-graph $P(V)=(\{*\},V)$ over the singleton set $\{*\}$. 
Then we find that (here we use that $P$ is fully faithful, see \cref{Pff})
$$\Vv(\bot,V)\cong \VGrph((\{*\},\bot),(\{*\},V)) = \VGrph(L(\{*\}),(\{*\},V)) = \Set(\{*\},\{*\})=\{id_{\{*\}}\}.$$
This shows that $\bot$ is indeed an initial object in $\Vv$.

Similarly, if $\Vv$ has a terminal object $\top$, the forgetful functor $U:\VGrph\to\Set$ has a fully faithful right adjoint $R$ which assigns to any set $X$ the graph over $X$ which has the terminal object in every component. If a fully faithful right adjoint $R$ for $U$ exists, then we can write $R(\{*\})=(\{*\},\top)$ and we find
$$\Vv(V,\top)\cong \VGrph((\{*\},V),(\{*\},\top)) = \VGrph((\{*\},V),R(\{*\})) = \Set(\{*\},\{*\})=\{id_{\{*\}}\},$$
which shows that $\top$ is a terminal object for $\Vv$.
\end{proof}

{
The following characterization of monomorphisms in $\VGrph$ can be found in \cite[Lemma 4.1]{KelLac}. We provide a proof for sake of completeness and also give a similar characterization for epimorphisms.}

\begin{proposition}\label{epimonoVGrph}
Let $\Vv$ be a category with a terminal and initial object. 
A morphism $\ul f=(f^0,f^1):\ul A\to \ul B$ in $\VGrph$ is a
\begin{enumerate}[(i)]
\item monomorphism if and only if $f^0:A^0\to B^0$ is injective and  for all $x,y\in A^0$, $f_{x,y}:A_{x,y}\to B_{fx,fy}$ is a monomorphism in $\Vv$.
\item epimorphism if and only if $f^0:A^0\to B^0$ is surjective and for all $x,y\in B^0$, the family $f_{x',y'}:A_{x',y'}\to B_{x,y}$, indexed by all $(x',y')\in A^0\times A^0$ for which $(fx',fy')=(x,y)$, is jointly epic in $\Vv$.
\end{enumerate}

\begin{proof}
It follows from \cref{Upreserves} that if $\ul f$ is a monomorphism, then $f^0$ is a monomorphism in $\Set$, hence injective. Fix any $x,y\in A^0$ and any pair of morphisms $g,h:V\to A_{x,y}$ in $\Vv$ such that $f_{x,y}\circ g=f_{x,y}\circ h$. Then consider the $\Vv$-graph $\ul V$, with set of objects $V^0=A^0$, $V_{x,y}=V$ and $V_{x',y'}=A_{x',y'}$ for all $(x',y')\in A^0\times A^0\setminus\{(x,y)\}$. Then we have morphisms $\ul g,\ul h:\ul V\to \ul A$, where $g^0$ and $h^0$ are the identity maps, $g_{x,y}=g$, $h_{x,y}=h$ and all other components of $\ul g$ and $\ul h$ are identities. Then clearly, $\ul f\circ \ul g=\ul f\circ \ul h$ and therefore $\ul g=\ul h$ as $\ul f$ is a monomorphism. In particular, $g=g_{x,y}=h_{x,y}=h$. We can conclude that $f_{x,y}$ is a monomorphism in $\Vv$.  Conversely, if $\ul f$ satisfies the conditions of the statement, and $\ul f\circ \ul g=\ul f\circ \ul h$ for some parallel morphisms $\ul g$ and $\ul h$, then since $f^0$ is injective, we already have $g^0=h^0$. Since moreover each component of $\ul f$ is a monomorphism in $\Vv$, also all components of $\ul g$ and $\ul h$ coincide.

Dually, it follows again from \cref{Upreserves} that if $\ul f$ is an epiomorphism, then $f^0$ is an epimorphism in $\Set$, hence surjective. Now fix $x,y\in B^0$, and any pair of morphisms $g,h:B_{x,y}\to V$ in $\Vv$ such that $g\circ f_{x',y'}=h\circ f_{x',y'}$ for all $(x',y')\in A^0\times A^0$ for which $(fx',fy')=(x,y)$. Then consider the $\Vv$-graph $\ul V$, with set of objects $V^0=B^0$, $V_{x,y}=V$ and $V_{x'',y''}=B_{x'',y''}$ for all $(x'',y'')\in B^0\times B^0\setminus\{(x,y)\}$. Then we have morphisms $\ul g,\ul h:\ul B\to \ul V$, where $g^0$ and $h^0$ are identity maps, $g_{x,y}=g$, $h_{x,y}=h$ and all other components of $\ul g$ and $\ul h$ are identities. Then clearly, $\ul g\circ \ul f=\ul h\circ \ul f$ and therefore $\ul g=\ul h$ as $\ul f$ is an epimorphism. In particular, $g=g_{x,y}=h_{x,y}=h$. We can conclude that the family $f_{x',y'}$ is a jointly epic in $\Vv$. The converse is again a direct consequence of the definitions.
\end{proof}

\end{proposition}

The category of $\Vv$-graphs inherits many nice properties from $\Vv$. Many results of this flavour have been proven throughout literature, but we just mention the following result, which is sufficient for our needs.

\begin{proposition}[{\cite[Proposition 4.4]{KelLac}}]
Let $\Vv$ be an arbitrary category. 
If $\Vv$ is locally presentable, then so is $\VGrph$.
\end{proposition}

\begin{remark}\label{limVGrph}
Recall from \cite[Proposition 2]{BCSW} (see also \cite[Section 3]{KelLac}, \cite[Proposition 4.14]{Vasi}) that colimits in $\VGrph$ can be build up explicitly from colimits in $\Set$ and $\Vv$, although not in a straightforward way.  As remarked in \cite[Proposition 4.14]{Vasi}, the situation of limits in $\VGrph$ is much easier, and since we will need those further on in the paper, let us describe these explicitly here.
Consider a small category $\Zz$ and a functor $F:\Zz\to \VGrph$. For any $Z\in\Zz$, we denote $FZ=(FZ^0,FZ^1)$. Take the limit $(L^0,\lambda^0_Z:L^0\to FZ^0)$ in $\Set$ of the composite functor $U\circ F:\Zz\to \Set$. We know from \cref{Upreserves} that $L^0$ will be the set of objects of $\lim F$ or more precisely that $U \lim F=(L^0,\lambda^0)$. Now fix any $x,y\in L^0$, and consider the functor $F_{x,y}:\Zz\to \Vv$ defined on an object $Z\in \Zz$ as $F_{x,y}(Z)=FZ_{\lambda^0_Z(x),\lambda^0_Z(y)}$ and on a morphism $f:Z\to Z'$ in $\Zz$ as $Ff_{\lambda^0_Z(x),\lambda^0_Z(y)}$. Then consider the limit $\lim F_{x,y}=(L^1_{x,y},\lambda^1_{x,y})$ in $\Vv$. From the universal properties of the considered limits, one easily deduces that $\lim F=((L^0,L^1),(\lambda^0,\lambda^1))$. 
\end{remark}

\subsection{The category of $\Vv$-categories}

From now on, let us suppose that $\Vv$ is a monoidal category, whose tensor product we denote by $\ot$ and monoidal unit by $I$. By Mac Lane's coherence theorem, without loss of generality, we can omit writing associativity and unitality constraints. Let us recall the (well-known) notion of a $\Vv$-enriched category, see e.g.\ \cite{Kelly}.

\begin{definition}
A {\em $\Vv$-enriched category} (or a {\em $\Vv$-category} for short) is a $\Vv$-graph $\ul A=(A^0,A^1)$ endowed with $\Vv$-morphisms (called the composition\footnote{As one can see from the definition of composition, one should interpret $A_{yx}$ as the object of morphisms ``from $y$ to $x$''. Traditionally, many authors use the reversed notation, however we believe the notation used here is more efficient for our needs.} or multiplication morphisms)
\begin{eqnarray*}
m_{xyz}: A_{xy}\ot A_{yz} \to A_{xz} \qquad j_x: I \to A_{xx},
\end{eqnarray*}
for all $x,y,z\in A^0$ satisfying the following axioms.
\[
\xymatrix{A_{xy} \ar[drr]^{\sf id}\ar[r]^-{\simeq} \ar[d]_{\simeq} & A_{xy} \ot I \ar[r]^-{{\sf id }\ot j_y} & A_{xy}\ot A_{yy} \ar[d]^{m_{xyy}} & A_{xy}\ot A_{yz} \ot A_{zw} \ar[r]^-{m_{xyz}\ot {\sf id}} \ar[d]_-{{\sf id} \ot m_{yzw}} & A_{xz}\ot A_{zw} \ar[d]^{m_{xzw}}\\
I \ot A_{xy} \ar[r]_-{j_x\ot {\sf id}} & A_{xx}\ot A_{xy} \ar[r]_-{m_{xxy}} & A_{xy} & A_{xy}\ot A_{yw} \ar[r]_-{m_{xyw}} & A_{xw}}.
\]
A {\em $\Vv$-functor} is a $\Vv$-graph morphism that preserves multiplication and unit. Explicitely $\ul f: (A^0,A^1,m,j)\to (B^0,B^1,m',j')$ consists of a function $f^0:A^0 \to B^0$ and a family of morphisms $f_{xy}:A_{xy}\to B_{fx,fy}$ in $\Vv$ for all $x,y\in A^0$, such that the following diagrams commute for all $x,y,z\in A^0$.
\[
\xymatrix{I \ar[r]^-{j_x} \ar[dr]_-{j'_{fx}} & A_{xx} \ar[d]^{f_{xx}} & A_{xy}\ot A_{yz} \ar[rr]^-{m_{xyz}} \ar[d]_{f_{xy}\ot f_{yz}} &  & A_{xz} \ar[d]^{f_{xz}}\\
& B_{fxfx}  & B_{fxfy}\ot B_{fyfz} \ar[rr]_-{m'_{fxfyfz}} & & B_{fxfz} }.
\]
One observes that $\Vv$-categories and $\Vv$-functors form a the category, which we denote by $\VCat$.
\end{definition}

\begin{remark}
One can of course also consider $\Vv$-natural transformations which turn $\VCat$ into a $2$-category. The same holds for the categories semi-Hopf categories and Hopf-categories which we consider below. 
\end{remark}

{
Recall that an algebra (or monoid) in a monoidal category $\Vv$, is an object $A$ endowed with a multiplication $m:A\ot A\to A$ and a unit $j:I\to A$ satisfying the usual associativity and unitality conditions. A morphism of between algebra $(A,m,j)$ and $(B,m',j')$ is a morphism $f:A\to B$ satisfying $m'\circ (f\ot f) = f\circ m$ and $j'=f\circ j$. We denote the category algebras and their algebra morphisms by $\Alg(\Vv)$. A coalgebra in $\Vv$, also called a $\Vv$-coalgebra, is an algebra in the opposite category $\Vv^{op}$, and we denote by $\Coalg(\Vv)=\Alg(\Vv^{op})^{op}$ the category of coalgebras in $\Vv$.}
Clearly, a $\Vv$-category with one object contains exactly the same data as an algebra in $\Vv$, and a $\Vv$-opcategory with one object contains the same data as a coalgebra in $\Vv$.

Suppose now that $\Vv$ has an initial object $\bot$ such that for any object $V$ in $\Vv$, we have $V\ot \bot\cong \bot\cong \bot\ot V$. We will refer to the property as saying that the initial object is preserved under tensoring. This is for example satisfied if the endofunctors of the form $-\ot V$ and $V\ot -$ preserve (finite) colimits for all $V$ in $\Vv$. Then to any object $V$ in $\Vv$ we can associate a $\Vv$-category $\ul V$ in the following way. As set of objects we take a two-element set $V^0=\{0,1\}$ and we define the Hom-objects $V_{x,y}$ as follows: $V_{0,0}=V_{11}=I$, $V_{0,1}=V$ and $V_{1,0}=\bot$. Then the identity on the monoidal unit $I$ induces the unit morphisms $j_0=j_1={\sf id}_{I}$. Composition is then obtained in an obvious way from the unit constraints of the monoidal structure on $\Vv$ and the universal property of the initial object. These observations lead directly to the following (known) result.

\begin{proposition}
\begin{enumerate}[(i)]
\item
The functor $P:\Vv\to \VGrph$ from \cref{Pff} lifts to a fully faithful functor $\PP:\Alg(\Vv)\to \VCat$ such that the following diagram commutes, where $\Alg(\Vv)$ denotes the category of algebras in $\Vv$, and the unadorned vertical arrows are the obvious forgetful functors.
\[
\xymatrix{
\Alg(\Vv) \ar[rr]^-\PP \ar[d] && \VCat \ar[d]\\
\Vv\ar[rr]^-P && \VGrph
}
\]
\item
If $\Vv$ has an initial object that is preserved under tensoring, then there is a faithful functor $\PP':\Vv\to \VCat$, sending $V$ to the two-object $\Vv$-category $\ul V$ defined above.
\end{enumerate}
\end{proposition}

Let us start by studying the forgetful functor to $\Set$, similar to \cref{Upreserves}.
\begin{proposition}\label{Uvcattosetpreserves}
Consider the forgetful functor $U:\VCat\to\Set$.
\begin{enumerate}[(i)]
\item If $\Vv$ has an initial object that is preserved under tensoring, then $U$ has a fully faithful left adjoint. 
\item If $\Vv$ has a terminal object, then $U$ has a fully faithful right adjoint.
\end{enumerate}
In particular, if $\Vv$ has an initial and terminal object, then $U$ preserves all limits and colimits, in particular monomorphisms and epimorphisms.
\end{proposition}

\begin{proof}
\ul{(i)}
Consider any set $X$. We define a $\Vv$-category $L$ with $L^0=X$ as set of objects by putting $L_{xx}=I$, the monoidal unit of $\Vv$, for all $x\in X$ and $L_{xy}=\bot$, the initial object of $\Vv$, for all $x,y\in X$ with $x\neq y$. Then $L$ becomes in a trivial way a $\Vv$-category and this construction gives a fully faithful left adjoint to $U$.

\ul{(ii)}.
If $\Vv$ has a terminal object $\top$, then define another $\Vv$-category $R$ with $R^0=X$ as set of objects by putting $R_{xy}=\top$, the terminal object, for all $x,y\in X$. Again $R$ becomes in a trivial way a $\Vv$-category and this construction provides a fully faithful right adjoint for $U$. 
\end{proof}

{
As for $\VGrph$, also the category $\VCat$ inherits many good properties from $\Vv$. Let us remark that \cite{BCSW} considers the more general setting of categories enriched in a bicategory, but we formulate the result directly in the $\Vv$-enriched setting we need here. Another relevant reference is \cite{Wolff}.
}

\begin{proposition}\label{Vcatlp}
Let $\Vv$ be a monoidal category.
\begin{enumerate}[(i)]
\item \cite[Proposition 3, Theorem 6, Theorem 7]{BCSW} Suppose that $\Vv$ is cocomplete and for any object $A$ in $\Vv$, endofunctors on $\Vv$ of the form $A\ot-$ or $-\ot A$ preserve colimits. Then $\VCat$ is cocomplete as well and the forgetful functor $U:\VCat\to\VGrph$ has a left adjoint and is moreover monadic.
\item \cite[Theorem 4.5]{KelLac} If $\Vv$ is closed monoidal and $\lambda$-locally presentable, then $\VCat$ is $\lambda$-locally presentable as well.
\end{enumerate}
\end{proposition}

\begin{remarks}\label{rem:freeVcat}
\begin{enumerate}
\item
At first look, it might appear that the condition on the preservation of colimits by endofunctors $A\ot-$ and $-\ot A$ in item (i) of the above Proposition is weaker than the closed monoidality in item (ii). However, in view of \cref{adjointlp} both conditions are equivalent if $\Vv$ is locally presentable.
\item
The left adjoint $T:\VGrph\to\VCat$ to the forgetful functor $U$, creates the ``free $\Vv$-category'' over a $\Vv$-graph, see for example in \cite{Wolff}. As it will be useful for what follows, let us recall this construction. In case we start from a one-object $\Vv$-graph, that is, just an object $V\in \Vv$, then $T(V)$ is a one-object $\Vv$-category, that is, just a monoid in $\Vv$, namely the free monoid over the object $V$. In general the functor $T$ commutes with the forgetful functor to $\Set$. Explicitly, for a given $\Vv$-graph $\ul A=(A^0,A^1)$, the free $\Vv$ category $T(\ul A)$ has the same set of objects $A^0$, and for any $x,y\in A^0$, we have 
\begin{equation}\eqlabel{TAexplicit}
T(\ul A)_{x,y}=\coprod_{n\in \NN, z_1,\ldots, z_n \in A^0} A_{x,z_1}\ot A_{z_1,z_2}\ot \cdots \ot A_{z_n,y}.
\end{equation}
In case $x=y$, $T(\ul A)_{x,x}$ has an additional component of the form $I$ (the monoidal unit), that induces the unit morphisms $j_x:I\to T(\ul A)_{x,x}$. Composition in (or multiplication of) the $\Vv$-category $T(\ul A)$ is obtained from the monoidal associativity constraints $$
\xymatrix{\left(A_{x,z_1}\ot \cdots \ot A_{z_n,y}\right)\ot \left(A_{y,u_1}\ot \cdots \ot A_{u_n,z}\right)\ar[r]^-\cong & A_{x,z_1}\ot \cdots \ot A_{z_n,y}\ot A_{y,u_1}\ot \cdots \ot A_{u_n,z}}$$
and the universal property of the coproduct (recall that coproducts are preserved under tensor product by assumption). 

For any $\Vv$-graph $\ul A$, the unit of the adjunction $\eta^A:\ul A\to UT(\ul A)$ is given by the canonical morphisms
$$\eta^A_{x,y}:A_{x,y}\to T(\ul A)_{x,y}$$
for all $x,y\in A^0$, which are induced by the fact that $T(\ul A)_{x,y}$ is defined as a coproduct with $A_{x,y}$ as one of the components (see \equref{TAexplicit}). On the other hand, given a $\Vv$-category $\ul A$, the counit of the adjunction $\epsilon^A: TU(\ul A)\to \ul A$ is induced by the composition in (or multiplication of) $\ul A$.

Remark that by construction of $T(\ul A)$ as a coproduct and the composition as defined above, we have that for any $x,y\in A^0$, the family of morphisms
$$\xymatrix{A_{x,z_1}\ot A_{z_1,z_2}\ot \cdots \ot A_{z_n,y} \ar[rr]^-{\eta_{x,z_1}\ot \cdots \ot\eta_{z_n,y}} && 
T(\ul A)_{x,z_1}\ot T(\ul A)_{z_1,z_2}\ot \cdots \ot T(\ul A)_{z_n,y} \ar[r]^-{m} & T(\ul A)_{x,y}
}$$
indexed by all $n\in \NN, z_1,\ldots,z_n\in A^0$ are jointly epic in $\Vv$, where $\eta:\ul A\to UT(\ul A)$ denotes the unit of the adjunction $(T,U)$ and $m$ denotes the obvious morphism induced by the composition in $T(\ul A)$. 
In case $x=y$, one also needs to add the unit morphisms $j_x: I\to T(\ul A)_{x,x}$ to this family.
\item
{ 
Let us now describe colimits in $\VCat$. Consider a functor $F:\Zz\to\VCat$ and its colimit $\colim F$ in $\VCat$ together with the canonical morphisms $\gamma_Z:FZ\to \colim F$ for all $Z\in \Zz$. Also consider the composite functor $UF:\Zz\to\VGrph$ and its colimit $\colim UF$ in $\VGrph$ with the canonical morphisms $\gamma'_Z:UFZ\to \colim UF$. By the universal property of $\colim UF$, we obtain a unique morphism $q':\colim UF\to U\colim F$ such that $q'\circ \gamma'_Z=U\gamma_Z$ for all $Z$ in $\Zz$. By the adjunction $(T,U)$, this induces a morphism of $\Vv$-categories $q:T\colim UF\to \colim F$, such that
\begin{eqnarray*}
U\gamma_Z &=& q' \circ \gamma'_Z \\
&=& Uq\circ \eta^{\colim UF} \circ \gamma'_Z\\
&=& Uq\circ UT\gamma'_Z \circ \eta_{UFZ}.
\end{eqnarray*}
In fact, $Uq$ is even a (regular) epimorphism in $\Vv$-graph. Indeed, one can follow the construction of \cite[Proposition 2.11]{Wolff} to see that (the underlying $\Vv$-graph of) $\colim F$ can be constructed as the coequalizer (of the underlying $\VGrph$-morphisms) of parallel morphisms 
$$\xymatrix{ TUT\colim UF \ar@<.5ex>[rr] \ar@<-.5ex>[rr] && T\colim UF}$$ 
induced by the morphisms $\epsilon^{TUFZ}, TU\epsilon^{FZ}: TUTUFZ\to TUFZ$ for all $Z$. 
This explains that $\colim F$ can be viewed as the ``largest quotient'' $$q:T \colim UF\cong \colim TUF\to \colim F$$ in $\VGrph$ that is endowed with a $V$-category structure such that the $\VGrph$ morphisms $U\gamma_Z$ as defined by the above equation are the underlying morphisms of $\VCat$-morphisms for any $Z\in \Zz$. Since we know moreover that $Uq$ is an epimorphism in $\Vv$-graph, the family of all components of $q$ landing in the same codomain are jointly epic by \cref{epimonoVGrph}. In the diagram below, this family is pre-composed with the joinly epic family of morphisms landing in components of the free $\Vv$-category $T\colim UF$, as discussed in point (2), and the jointly epic family $\gamma'$ consisting of the components of the colimit $\colim UF$.
\[
\xymatrix{
\bigotimes_{i=1}^n (UFZ_i)_{x_{i1},x_{i2}} \ar[d]_-{\bigotimes_{i=1}^n (\gamma'_{Z_i})_{x_{i1},x_{i2}}} \ar@{=}[rr]
&& \bigotimes_{i=1}^n (UFZ_i)_{x_{i1},x_{i2}} \ar[d]_-{\bigotimes_{i=1}^n\eta^{UFZ_i}_{x_{i1},x_{i2}}} \ar@/^10pc/[ddd]_-{\bigotimes_{i=1}^n (\gamma_{Z_i})_{x_{i1},x_{i2}}}\\
\bigotimes_{i=1}^n (\colim UF)_{x'_{i-1},x'_i} \ar[d]_-{\bigotimes_{i=1}^n\eta^{\colim UF}_{x'_{i-1},x_i}} 
&& \bigotimes_{i=1}^n (UTUFZ_i)_{x_{i1},x_{i2}} \ar[d]_-{\bigotimes_{i=1}^n (UT\gamma'_{Z_i})_{x_{i1},x_{i2}}}\\
\bigotimes_{i=1}^n UT(\colim UF)_{x'_{i-1},x'_i} \ar[d]_-{m_{x'_0\cdots x'_n}}  \ar@{=}[rr]
&& \bigotimes_{i=1}^n UT(\colim UF)_{x'_{i-1},x'_i} \ar[d]_-{\bigotimes_{i=1}^n Uq_{x'_{i-1},x'_i}} \\
UT\colim UF_{x'_0,x'_n} \ar[d]_-{Uq} 
&& \bigotimes_{i=1}^n U\colim F_{x_{i-1},x_{i}} \ar[d]_-{m_{x_0\cdots x_n}}  \\
U\colim F_{x_0,x_n} \ar@{=}[rr] && U\colim F_{x_0,x_n}
}
\]
By the commutativity of the above diagram, we can therefore conclude that for any functor $F:\Zz\to \VCat$, and any $x,y\in (\colim F)^0$ the family of $V$-morphisms
$$
\xymatrix{
(UFZ_1)_{x_{11},x_{12}}\ot (UFZ_2)_{x_{21},x_{22}} \cdots \ot (UFZ_n)_{x_{n1},x_{n2}} \ar[d]^-{(\gamma_{Z_1})_{x_{11}x_{12}}\ot (\gamma_{Z_1})_{x_{21}x_{22}}\ot\cdots \ot (\gamma_{Z_1})_{x_{n1}x_{n2}}} \\
(\colim F)_{xx_1} \ot (\colim F)_{x_1x_2} \ot \cdots \ot (\colim F)_{x_{n-1}y}  \ar[d]^-{m_{xx_1x_2\cdots x_{n-1}y}} \\ 
(\colim F)_{xy}
}
$$
where we vary over all $n\in \NN_0$, $Z_1,\ldots,Z_n\in \Zz$, $x_{i1},x_{i2}\in (UFZ_1)^0$ such that $\gamma^0_{Z_i}(x_{i2})=\gamma^0_{Z_{i+1}}(x_{i+1,1})=x_i$, $\gamma^0_{Z_1}(x_{11})=x$ and $\gamma^0_{Z_n}(x_{n2})=y$ is jointly epic (in $\Vv$). 
}
\end{enumerate}
\end{remarks}

\begin{definition}
Let $\Vv$ be a monoidal category, then we denote by $\Vv^{rev}$ the monoidal category which has the same underlying category as $\Vv$, but whose tensor product is reversed. More precisely, we have that for any two objects $A,B$ in $\Vv$, 
$$A\ot^{rev} B := B\ot A.$$
and similarly for morphisms.

Let $\ul A$ be a $\Vv$-category.  Recall (see e.g. \cite{JanKel}) that the opposite $\Vv$-category of $\ul A$ is defined as the $\Vv^{rev}$-category $\ul A^{op}$ whose underlying $\Vv$-graph is the opposite graph of $\Aa$ (that is, the set of objects is $A^0$ and for each $x,y\in A^0$, we have $A^{op}_{xy}=A_{yx}$) endowed with compositions given by 
$$\xymatrix{A^{op}_{x,y}\ot^{rev} A^{op}_{y,z} = A_{z,y}\ot A_{y,x} \ar[rr]^-{m_{zyx}} && A_{z,x} = A^{op}_{x,z}
}$$
and the same unit morphisms as $\ul A$. 

Similarly, if $\ul f:\ul A\to \ul B$ is a $\Vv$-functor, then we define $\ul f^{op}:\ul A^{op}\to \ul B^{op}$ by $(f^{op})^0=f^0$ and $f^{op}_{x,y}=f_{y,x}$.
\end{definition}

The following result is immediate

\begin{proposition}\label{oppositecat}
There is an endofunctor $(-)^{op}:\VCat\to {\textsf{$\Vv^{rev}${\sf-Cat}}}$ which sends a $\Vv$-category to its opposite $\Vv^{rev}$-category. This functor is clearly isomorphism of categories, whose inverse functor is constructed in the same way, since $(\Vv^{rev})^{rev}=\Vv$ as a monoidal category.
\end{proposition}

{
The $2$-functor from \cref{Grphfunctor} induces another $2$-functor in our current setting.
}

\begin{proposition}\label{Catfunctor}
Any lax monoidal functor $F:\Vv\to\Ww$ induces a functor $\FCat:\VCat\to\WCat$. Similarly, if $G:\Vv\to\Ww$ is another lax monoidal functor, then any monoidal natural transformation $\alpha:F\Rightarrow G$ induces a natural transformation $\FCat\Rightarrow\GCat$. In fact, we obtain a $2$-functor $\LMonCat\to \Cat$, where $\LMonCat$ denotes the $2$-category of monoidal categories, lax monoidal functors and monoidal natural transformations. It follows that any lax monoidal adjunction between $\Vv$ and $\Ww$ also induces an adjunction between $\VCat$ and $\WCat$.
\end{proposition}

\begin{proof}
Let us denote by $I$ (resp. $J$) the monoidal unit of $\Vv$ (resp.\ $\Ww$) and the monoidal product in both cases by $\ot$. We also denote the monoidal structure on $F$ by $\phi^0:J\to FI$ and $\phi^2:F-\ot F- \to F(-\ot -)$. Consider any $\Vv$-category $\ul A$. We already know by \cref{Grphfunctor} that $F(\ul A)$ is a $\Ww$-graph. It now suffices to observe that the morphisms
$$Fm_{xyz}\circ \phi^2_{A_{xy},A_{yz}}  :FA_{x,y}\ot FA_{y,z} \to FA_{x,z},\ Fj_x\circ \phi^0 : J\to FA_{x,x}$$
endow $F{\ul A}$ with a $\Ww$-category structure. Furthermore, given any morphism $\ul f:\ul A\to \ul B$ of $\Vv$-categories, one can then check that $\FGrph(\ul f)$ as defined in \cref{Grphfunctor} is a morphism of $\Ww$-categories with respect to the $\Ww$-category structure defined here. This defines the functor $\FCat$.
Finally, one observes that if $\alpha:F\Rightarrow G$ is a monoidal natural transformation, then $\alpha-\Grph$ as defined in \cref{Grphfunctor} is in fact a natural transformation from $\FCat$ to $\GCat$.
\end{proof}

\begin{remark}\label{remoppositecat}
Let $\Vv$ be a braided monoidal category, whose braiding we denote by $\sigma$. Then the identify functor and braiding induce a strong monoidal functor $(id,\sigma):\Vv^{rev}\to \Vv$. This functor is moreover an monoidal isomorphism, with inverse $(id,\sigma^{-1})$, where $\sigma^{-1}$ denotes the inverse braiding. Remark that if $\Vv$ is symmetric then both functors coincide. Specifying \cref{Catfunctor} to this monoidal functor, we find that for a braided monoidal category $\Vv$, there is an isomorphism of categories ${\textsf{$\Vv^{rev}${\sf-Cat}}}\to \VCat$. Combining this isomorphism with the isomorphism from \cref{oppositecat}, we obtain a new isomorphism
$$(-)^{op}:\VCat\to\VCat$$
which sends a $\Vv$-category $\ul A$ to its opposite $\Vv$-category, which, by abuse of notation, we still denote as $\ul A^{op}$. It has the opposite $\Vv$-graph as underlying $\Vv$-graph and compositions defined by
$$\xymatrix{A^{op}_{x,y}\ot A^{op}_{y,z} = A_{y,x}\ot A_{z,y} \ar[rr]^-\sigma && A_{z,y}\ot A_{y,x} \ar[rr]^-{m_{zyx}} && A_{z,x} = A^{op}_{x,z}
}$$
\end{remark}

\section{Cofree coalgebras}

\subsection{Existence of cofree coalgebras}

Let now $\Vv$ be a symmetric monoidal category, where we denote the symmetry by $\sigma$.
Then it is well-known that the category $\Coalg(\Vv)$ of coalgebras in $\Vv$ inherits in a natural way a (symmetric) monoidal structure from $\Vv$. Moreover, let us recall the following result which tells that $\Coalg(\Vv)$ also inherits the property of being locally presentable from $\Vv$.

\begin{proposition}[{\cite[2.7, 3.2]{Porst}}]\label{coalglp}
Let $\Vv$ be a closed symmetric monoidal and locally presentable category. Then the following assertions hold.
\begin{enumerate}[(i)]
\item The category $\Coalg(\Vv)$ is closed symmetric monoidal and locally presentable. 
\item The forgetful functor $U:\Coalg(\Vv)\to \Vv$ is strict symmetric monoidal and comonadic, in particular it has a right adjoint $T^c:\Vv\to \Coalg(\Vv)$ which creates the ``cofree coalgebra'' over an object in $\Vv$.  
\end{enumerate}
\end{proposition}

{ 
More explicitly, the ``cofree coalgebra'' over an object $V$ in a monoidal category $\Vv$, is a coalgebra $T^c(V)$ in $\Vv$ endowed with a $\Vv$-morphism $p:T^c(V)\to V$ such that for any other coalgebra $C$ in $\Vv$ and any $\Vv$-morphism $\gamma:C\to V$, there exists a unique coalgebra morphism $u:C\to T^c(V)$ such that $\gamma=\tau \circ u$.
\[
\xymatrix{
T^c(V) \ar[drr]^-p \\
&& V \\
C \ar[urr]_-\gamma \ar@{.>}[uu]^u
}
\]
}

\begin{remark}
In contrast to free algebras in $\Vv$, or more generally, free $\Vv$-categories (see Remark \ref{rem:freeVcat}), cofree coalgebras are usually more difficult to construct explicitly. Indeed, naively one would try to make a dual construction of free algebras by considering the product $T^0(V):=\prod_{n\in \NN} V^{\ot n}$. However, since usually infinite products are not preserved by tensoring with a fixed object in $\Vv$, there is no way in general to endow $T^0(V)$ with a suitable comultiplication. In fact, or at least to the knowledge of the authors, an explicit construction of cofree coalgebra over an object $V$ in $\Vv$ is only known in particular cases, such as $\Vv=\Mod_k$ with $k$ a field (see \cite{BloLer}) or a commutative ring (see \cite{Fox}). Our aim is to generalize this construction to the case of a monoidal category whose monoidal product is sufficiently exact in the sense pointed out in the next subsection. 
\end{remark}

\subsection{Very flat monoidal products}\selabel{veryflat}

The following terminology of ``(very) flat monoidal product''  is inspired by the module-theoretic case and will be essential futher in this paper.

\begin{definition}
Let $\Vv$ be a monoidal category and $A$ an object in $\Vv$. We say that $A$ is (left) {\em flat} if the endofunctor $-\ot A$ on $\Vv$ preserves monomorphisms.

We say that $A$ is (left) {\em very flat} if the endofunctor $-\ot A$ on $\Vv$ preserves arbitrary jointly monic families. This means, for any jointly monic family of morphisms $\beta_i:B\to B_i$ in $\Vv$, we have that the family $\beta_i\ot {\sf id}_A:B\ot A\to B_i\ot A$
is also jointly monic.

Similarly, we can consider right (very) flat objects. If $\Vv$ is symmetric (or even braided) left and right (very) flatness are equivalent. Since the categories we consider are symmetric we will not distinguish between left and right from now on.

If $\Vv$ is a monoidal category such that all objects are (very) flat, we say that $\Vv$ is has a {\em (very) flat monoidal product}.
\end{definition}

Although flat modules (that is, flat objects in a category of (bi)modules over a ring) are a standard notion, what we called ``very flat'' seems to be un(der)studied in literature. It is well-known (and can also be seen from Lemma \ref{tensorjointlymonic}(i) below) that a functor preserves jointly monic families if it preserves monomorphisms and products, in particular if it preserves all limits. The later conditions is very strong for a functor of the form $-\ot A:\Mod_k\to\Mod_k$, where $A$ is a $k$-module, since it expresses that the module is finitely generated and projective. The next examples show that ``very flatness'' is a much weaker condition for modules, being strictly in between flatness and projectivity.

\begin{examples}
\begin{enumerate}
\item Obviously, every very flat module is flat.
\item It is well-known that the $\ZZ$-module $\QQ$ is flat, however it is not very flat. Indeed, consider the jointly monic family of morphisms $\pi_p:\ZZ\to \ZZ_p$ for all prime numbers $p$. Since for any $p$, we have that $\ZZ_p\ot_\ZZ \QQ$ is the zero module, the family $\pi_p\ot_\ZZ\QQ$ is not jointly monic, hence $\QQ$ is not very flat as $\ZZ$-module.
\item Any locally projective module (in the sense of Zimmermann-Huisgen \cite{ZH}, called weakly locally projective in \cite{Ver:local}) is very flat, in particular every projective module is very flat. Indeed, let $k$ be a (commutative) ring and $P$ a locally projective $k$-module. Now consider a jointly monic family of $k$-module morphisms $\alpha_j:A\to A_j$. We claim that $P\ot_k\alpha_j:P\ot_k A\to P\ot_k A_j$ is again jointly monic. Indeed, consider any element $p\ot a \in P\ot_k A$ (summation understood) such that $p\ot \alpha_j(a)=0$ for all $j$. Then for any $f\in P^*$, we have that $f(p)\alpha_j(a)=\alpha_j(f(p)a)=0$ for all $j$. Since the $\alpha_j$ were jointly monic, we find that $f(p)a=0$ for all $f\in P^*$. Then the local projectivity of $P$ implies that $p\ot_k a=0$ (see e.g. \cite[Theorem 2.15]{Ver:local}).
{ Since there exist locally projective modules that are not projective (see \cite{ZH}), not every very flat module is projective\footnote{One could however wonder whether local projectivity is equivalent to very flatness.}.}
\item
From the above examples, it follows that the categories of projective modules and of locally projective modules over a commutative ring $k$, form a symmetric monoidal category with very flat tensor product. Similarly, modules over a triangular bialgebra (over a commutative ring $k$) that are locally projective over the base ring form a monoidal category with very flat tensor product (since the tensor product in this case is just the $k$-tensor product). 
Also $\Set$ with the cartesian product is a monoidal category with very flat monoidal product.
\end{enumerate}
\end{examples}

The following properties of very flatness will be crucial in what follows.

\begin{lemma}\label{tensorjointlymonic}
\begin{enumerate}
\item
Let $\Vv$ be a monoidal category and $A$ an object in $\Vv$. Then $A$ is very flat if and only if $A$ is flat and for any family objects $(B_i)_{i\in I}$ in $\Vv$, the canonical mophism $\jmath$ that makes the following diagrams commutative for all $i\in I$
$$\xymatrix{\left(\prod_{i\in I} B_i\right)\ot A \ar[dr]_-{\pi_i\ot {\sf id}_A} \ar[rr]^-\jmath && \prod_{i\in I}(B_i\ot A)  \ar[dl]^-{\pi'_i} \\
& B_i \ot A
}$$
is a monomorphism. Here we denoted by $\pi_i$ and $\pi'_i$ the canonical projections on the $i$th component of the considered products.
\item Let $\Vv$ be a monoidal category with very flat monoidal product. If $(\alpha_i:A\to A_i)_{i\in I}$ and $(\beta_j:B\to B_j)_{j\in J}$ are jointly monic families of morphisms, then $(\alpha_i\ot \beta_j:A\ot B\to A_i\ot B_j)_{(i,j)\in I\times J}$ is also jointly monic. In other words, tensor products of jointly monic families are again jointly monic.
\item
Let $\Vv$ be a monoidal category with very flat monoidal product.
Then for any two families $(A_i)_{i\in I}$ and $(B_j)_{j\in J}$ of objects in $\Vv$ the canonical morphism $\iota$ defined by the commutativity of the following diagram for all $i\in I$ and $j\in J$, is a monomorphism.
\[
\xymatrix{
\left(\prod_{i\in I} A_i\right)\ot \left(\prod_{j\in J} B_j\right) \ar[dr]_-{\pi_i\ot \pi_j} \ar[rr]^-\iota 
&& \prod_{(i,j)\in I\times J} (A_i\ot B_j) \ar[dl]^-{\pi_{i,k}}\\
& A_i\ot B_j
}\]
Here $\pi_i$, $\pi_j$ and $\pi_{i,j}$ denote the canonical projections of the considered products in $\Vv$.
\end{enumerate}
\end{lemma}

\begin{proof}
\ul{(1)}. Suppose that $A$ is very flat, then we know that $A$ is flat. Furthermore, the projections $\pi_i$ and $\pi'_i$ form jointly monic families. Therefore, the commutativity of the above diagram tell us that $\jmath$ is a monomorphism if and only if $\pi_i\ot {\sf id}_A$ is a jointly monic family, which is the case since $A$ is very flat. Conversely, consider a jointly monic family $\beta_i:B\to B_i$. Then the map $\beta=\prod \beta_i:B\to \prod_{i\in I} B_i$ is a monomorphism and since $A$ is flat $\beta\ot {\sf id}_A$ is also a monomorphism. We then have the following commutative diagram
\[
\xymatrix{
B\ot A \ar[d]_{\beta\ot {\sf id}_A} \ar[rr]^-{\beta_i\ot {\sf id}_A} && B_i\ot A\\
\left(\prod_{i\in I} B_i\right)\ot A \ar[rr]^-{\jmath} && \prod_{i\in I}(B_i\ot A) \ar[u]_-{\pi'_i}
}
\]
Since $\beta\ot {\sf id}_A$ and $\jmath$ are monic and $\pi'_i$ are jointly monic, we find that $\beta_i\ot {\sf id}_A$ are also jointly monic, hence $A$ is very flat.\\
\ul{(2)}. By very flatness, the morphisms $(\alpha_i\ot {\sf id}_B)_{i\in I}$ and $({\sf id}_{A_i}\ot \pi_j)_{j\in J}$ are also jointly monic families. 
Since $\alpha_i\ot \beta_j=(\alpha_i\ot B)\circ (A_i\ot \beta_j)$, we find that the family $(\alpha_i\ot \beta_j)_{(i,j)\in I\times J}$ being a composition of jointly monic families is again jointly monic.\\
\ul{(3)}. Observe that $\iota$ is a monomorphism if and only if $(\pi_i\ot \pi_j)_{(i,j)\in I\times J}$ is a jointly monic family. 
Classical properties of a product tell us that $(\pi_i)_{i\in I}$ and $(\pi_j)_{j\in J}$ are jointly monic families hence the statement follows by part (2).
\end{proof}

\subsection{The construction of cofree coalgebras}\selabel{cofree}

In the next Theorem, we will use the following notation. 
Let $(C,\Delta,\epsilon)$ be a coalgebra in $\Vv$. Then we define for all $n\in \NN$ a morphism $\Delta^n:C\to C^{\ot n}$ as follows:
$\Delta^0=\epsilon$ and for all $n>0$:
$$\Delta^n=(\Delta^{n-1}\ot {\sf id}_C)\circ \Delta,$$
in particular, $\Delta^1={\sf id}_C$ and $\Delta^2=\Delta$.

{
We are now ready to formulate and proof a first main theorem, concerning the existence of cofree coagebras. It is worthwhile to remark that in contrast to Proposition~\ref{coalglp}, we do not need to require that the monoidal product of $\Vv$ is symmetric.
}

\begin{theorem}\label{cofreeconstruction}
Let $\Vv$ be a closed monoidal and locally presentable category, with very flat monoidal product.
Let $V$ be an object of $\Vv$ and consider the product $\prod_{k\in \NN} V^{\ot k}$ where $V^{\ot 0}=I$, the monoidal unit. Denote by 
$\pi_n:\prod_{k\in \NN} V^{\ot k}\to V^{\ot k}$ the canonical projection on the 
$k$-th component of the product.

Then the cofree coalgebra $C=T^c(V)$ over an object $V$ in $\Vv$ is the largest subobject $c:C\to \prod_{k \in \NN} V^{\ot k}$ in $\Vv$, such that there exists a morphism $\Delta:C\to C\ot C$ satisfying $(\pi_k\ot \pi_l)\circ (c\ot c)\circ \Delta=\pi_{k+l}\circ c$ for all $k,l\in \NN$. 
\begin{equation}\label{propCc}
\xymatrix{
C \ar[d]^\Delta \ar[r]^-c & \prod_{k\in \NN}V^{\ot k} \ar[r]^-{\pi_{k+l}} & V^{\ot(k+l)} \ar[d]^\cong  \\
C\ot C \ar[r]^-{c\ot c} & \left(\prod_{k\in \NN}V^{\ot k}\right)\ot \left(\prod_{k\in \NN}V^{\ot k}\right) \ar[r]^-{\pi_k\ot \pi_l} 
& V^{\ot k}\ot V^{\ot l}
}
\end{equation}
The counit of the forgetful-cofree adjunction is then given by $p=\pi_1\circ c:C\to V$ and the morphisms $p^{\ot k}\circ \Delta^k:C\to V^{\ot k}$ are jointly monic.
\end{theorem}

\begin{proof}
Consider the family $\Cc$ of all sub-objects $c:C\to \prod_{k\in \NN}V^{\ot k}$ in $\Vv$, endowed with a morphism $\Delta:C\to C\ot C$ satisfying \eqref{propCc}.
We claim that such $\Delta$ is always coassociative and counital by means of $\epsilon_C=\pi_0\circ c$, i.e.\ $(C,\Delta_C,\epsilon_C)$ is a coalgebra. To prove the left counitality consider $(\epsilon_C\ot {\sf id})\circ \Delta:C\to I\ot C\cong C$ and compose these with the jointly monic family $(\pi_k\circ c)_{n\in \NN}$. Then the commutativity of \eqref{propCc} implies that 
$$\pi_k\circ c\circ (\epsilon_C\ot {\sf id})\circ \Delta = (\pi_0\ot \pi_k)\circ (c\ot c)\circ \Delta = \pi_k\circ c$$
for any $k\in \NN$, hence $(\epsilon_C\ot {\sf id})\circ \Delta={\sf id}_C$. The right counitality is proven in the same way. To prove the coassociativity, consider the parallel pair of morphisms $(\Delta\ot{\sf id})\circ \Delta, ({\sf id}_C\ot \Delta)\circ \Delta:C\to C\ot C\ot C$ and compose them with $(\pi_k\ot \pi_l\ot \pi_m)\circ (c\ot c\ot c)$. Using twice the commutativity of the diagram \eqref{propCc}, we find that both compositions equal $\pi_{k+l+m}\circ c$. Since the monoidal product is very flat, the family $((\pi_k\ot \pi_l\ot \pi_m)\circ (c\ot c\ot c))_{k,l,m\in\NN}$ is jointly monic being a tensor product of jointly monic families (see \cref{tensorjointlymonic}), hence $\Delta$ is coassociative.

Now consider any coalgebra $(C,\Delta,\epsilon)$, and $\gamma:C\to V$ a morphism in $\Vv$. Then for any $n\in \NN$, we define
$$\gamma_k:\xymatrix{C \ar[r]^-{\Delta^k} & C^{\ot k} \ar[r]^-{\gamma^{\ot k}} & V^{\ot k} }.$$
Consequently, we obtain a map $\ol\gamma=(\gamma_k)_{k\in \NN} :C\to \prod_{k\in \NN} V^{\ot k}$. Since $\Vv$ is locally presentable, we have a strong epi-mono factorization for $\ol\gamma$ (see \cref{epi-mono}):
$$\ol\gamma:\xymatrix{C\ar[r]^-e & C' \ar[r]^-m &\prod_{k\in \NN}V^{\ot k}}.$$
We claim that $(C', \Delta', \pi_0\ot m)$ is again a coalgebra and $e$ is a coalgebra morphism. Indeed, since $e$ is a strong epimorphism and $\iota\circ (m\ot m)$ is a monomorphism (remark that since the monoidal product is (very) flat, $m\ot m$ is a monomorphism as $m$ is one), we obtain the existence of a morphism $\Delta'$ filling the following commutative diagram. Recall, that the unique map $\iota$ such that $\pi_{p,q}\circ \iota=\pi_p\ot \pi_q$  is a monomorphism as in \cref{tensorjointlymonic}.
\[
\xymatrix{
C \ar[d]_\Delta \ar[r]^-e & C' \ar@{.>}[d]^{\Delta'} \ar[rr]^-m && \prod_{k\in \NN}V^{\ot k} \ar[d]^-D \ar[r]^-{\pi_{p+q}} & V^{\ot (p+q)} \ar[d]^\cong \\
C\ot C \ar[r]^-{e\ot e} & C'\ot C' \ar[r]^-{m\ot m}& \prod_{k\in \NN}V^{\ot k}\ot \prod_{k\in \NN}V^{\ot k} \ar[r]^-\iota 
& \prod_{p,q\in \NN}(V^{\ot p}\ot V^{\ot q}) \ar[r]^-{\pi_{p,q}} & V^{\ot p} \ot V^{\ot q}
}
\]
The morphism $D$ in the above diagram is defined by the universal property of the product in its domain.
Then we see that $(m:C'\to \prod V^{\ot n},\Delta')$ is an element of $\Cc$, therefore $(C',\Delta',\pi_0\circ m)$ is a coalgebra by the first part of the proof. The morphism $e$ is a coalgebra morphism by the commutativity of the left square in the above diagram. We conclude that for any morphism $\gamma:C\to V$ with $C$ a coalgebra, we obtain a factorisation
$$\gamma = \pi_1\circ \ol\gamma = \pi_1 \circ m\circ e$$
where by construction, $e:C\to C'$ is a coalgebra morphism and $(C',m)$ is an object of $\Cc$.

We are now ready to construct the cofree coalgebra over $V$. 
Since we assumed that $\Vv$ is locally presentable, $\Vv$ is in particular well-powered, and hence $\Cc$ is a set (and not a proper class). 
{ Consider the coproduct $\coprod_\Cc C$ over all objects $(C,c)$ in $\Cc$, which is, being a colimit of coalgebras in $\Vv$, again a coalgebra in $\Vv$, whose comultiplication we denote by $\Delta$.

The morphisms $c: C\to \prod_{k\in \NN}V^{\ot k}$ induce via the universal property of the coproduct a morphism  $[c]_{\Cc}:\coprod_\Cc C\to \prod_{k\in \NN}V^{\ot k}$, and define $\gamma=\pi_1\circ [c]_\Cc$. Then we can apply the procedure from the above paragraph. To this end, consider the morphisms $\gamma_k := \gamma^{\ot k}\circ \Delta^k$. We claim that the induced morphism $\ol \gamma = (\gamma_k)_{k\in \NN}:\coprod_\Cc C\to \prod_{k\in \NN}V^{\ot k}$ equals $[c]_\Cc$. To this end, we consider the following diagram, in which we precomposed both morphisms with the canonical monomorphism $\iota_C: C\to \coprod_\Cc C$ and the canonical epimorphism $\pi_n:\prod_{k\in \NN}V^{\ot k}\to V^{\ot n}$ for an arbitrary $C$ in $\Cc$ and $n$ in $\NN$.
\[
\xymatrix{
C \ar[rr]^-{\iota_C} \ar[dd]_-{\iota_C} \ar[drr]^-{\Delta_C^n} \ar[ddrr]_-{c} && \coprod_{\Cc} C \ar[rrr]^-{\Delta^{n}}   &&& (\coprod_\Cc C)^{\ot n} \ar[dl]^-{([c]_\Cc)^{\ot n}} \ar[dd]^-{\gamma^{\ot n}} \\
&& C^{\ot n} \ar[urrr]^-{(\iota_C)^{\ot n}} \ar[rr]_-{c^{\ot n}} && \left(\prod_{k\in \NN} V^{\ot k}\right)^{\ot n} \ar[dr]_-{(\pi_1)^{\ot n}} \\
\coprod_\Cc C \ar[rr]_-{[c]_\Cc} && \prod_{k\in \NN} V^{\ot k} \ar[rrr]_-{\pi_n} &&& V^{\ot n}
}
\]
In this diagram, the triangles and quadrangles commute by definition of the various maps (and universal property of the coproduct). The inner pentangle commutes by iteration of \eqref{propCc} which is satisfied by all elements of $\Cc$. 

By the above, we then find that $\ol\gamma=\coprod c$ admits a strong epi-mono factorization of the form}
$$\xymatrix{\coprod_\Cc c:\ \coprod_{\Cc} C \ar[r]^-{e'} & C(V) \ar[r]^-{m'} &  \prod_{k\in \NN}V^{\ot k}}$$
where $C(V)$ is a coalgebra and $e'$ is a coalgebra morphism.  By construction, $(C(V),m')$ satisfies \eqref{propCc} and therefore is itself an element of $\Cc$. For any (other) $(C,c)$ in $\Cc$, we then find that $c=[c]_\Cc\circ \iota_C=m'\circ e'\circ \iota_C$ is a monomorphism and therefore $e'\circ \iota_C : C\to C(V)$ is a monomorphism. This shows that $C(V)$ is indeed the largest element of $\Cc$.

If we now consider again any coalgebra $C$ with a morphism $\gamma:C\to V$, then then the results above show that $\gamma$ factors uniquely through $C(V)$ via a coalgebra morphism (the morphism $e'\circ \iota_{C'}\circ e$ in the diagram below, where $\iota_{C'}$ is the canonical monomorphism of the coproduct), so $C(V)$ satisfies the universal property of the cofree coalgebra  (see paragraph after \cref{coalglp}).
\[
\xymatrix{
C \ar[rrrr]^-\gamma \ar[dr]_-e &&&& V\\
& C' \ar[urrr]^-m \ar[r]_-{\iota_{C'}} & \coprod_{D\in \Cc} D \ar[r]^-{e'} & C(V) \ar[ru]_-{m'}
}
\]

Since by construction, $c:C(V)\to \prod_{k\in \NN}V^{\ot k}$ is a subobject such that $\pi_k\circ c=p^{\ot k}\circ \Delta^k$, these form a jointly monic family indexed by $k\in \NN$.
\end{proof}

Since the forgetful functor $U:\Coalg(\Vv)\to \Vv$ is comonadic, it preserves (and, in fact, creates) colimits. Limits in a category of coalgebras are, however, difficult to describe, although it follows from \cref{coalglp} that they exist. In case of a very flat monoidal product, we can apply the explicit description of the cofree coalgebra given above to give also a more explicit description of limits of coalgebras.

\begin{corollary}\label{limitsincoalg}
Let $\Vv$ be a closed monoidal and locally presentable category, with very flat monoidal product.
Let $F:\Zz\to \Coalg(\Vv)$ be a functor, where $\Zz$ is a small category and consider its limit $L=\lim F$ as well as the limit $L'=\lim UF$ where $U:\Coalg(\Vv)\to \Vv$ is the forgetful functor together with the canonical projections $\pi'_Z:L'=\lim UF\to UFZ$ for any $Z$ in $\Zz$.
\begin{enumerate}[(i)]
\item  $L$ can be constructed as the largest subcoalgebra of the cofree coalgebra $T^c(L')$, such that for any object $Z$ in $\Zz$ the composition 
$$U\pi_Z:\xymatrix{UL \ar[r]^-{Ui} & UT^c L' \ar[r]^-p  & L' \ar[r]^-{\pi'_Z} & UFZ}$$
is the underlying $\Vv$-morphism of a a coalgebra morphism $\pi_Z:L\to FZ$, where $i$ is the monomorphism that turns $L$ into a subobject of $T^cL'$ and $p:UT^cL'\to L'$ is the counit of the adjunction $(U,T^c)$. The morphism $\pi_Z$ formed this way is exactly the canonical projection of the limit $L=\lim F$ on the component $FZ$.
\item The family of morphisms 
$$(\pi_{Z_1}\ot \cdots \ot \pi_{Z_n})\circ \Delta_L^n: L\to FZ_1\ot \cdots \ot FZ_n$$
indexed by $n\in \NN$ and objects $Z_1,\ldots,Z_n\in \Zz$, is jointly monic in $\Vv$.
\end{enumerate}
\end{corollary}

\begin{proof}
\ul{(i)}. This part can be proven by applying the same technique as in the proof of \cref{cofreeconstruction}:  One considers the family $\Ll$ of all sub-coalgebras of $T^cL'$ satisfying a property as in the statement. Next, one constructs the coproduct (in $\Vv$ or directly in $\Coalg(\Vv)$, which gives the same result) of all elements in $\Ll$. Performing a strong epi-mono factorization of the incuded morphism from the coproduct in $T^cL'$, we obtain this way the ``largest subcoalgebra'' $L$ of $T^c L'$ satisfying the property of the statement.\footnote{Alternatively, one can use a similar reasoning as in Remark~\ref{rem:freeVcat}(3), and construct $L$ as an equalizer of parallel morphisms $T^c L'\rightrightarrows T^cU T^c L'$, induced by the unit of the adjunction $(U,T^c)$.} This object can then easily be verified to satisfy the universal property of the limit.

\ul{(ii)}. 
By part (i), we know that there is a monomorphism $i: \lim F\to T^c\lim UF$ in $\Vv$. From the construction of the cofree coalgebra in \cref{cofreeconstruction}, we know that the family of morphisms 
$$\xymatrix{T^c\lim UF\ar[r]^-{\Delta^n} & (T^c\lim UF)^{\ot n} \ar[r]^-{p^{\ot n}} & (\lim UF)^{\ot n}}$$ is jointly monic. Moreover, by general properties of categorical limits, the family of projections $\pi'_Z:\lim UF\to UFZ$ over all objects $Z\in \Zz$ is jointly monic in $\Vv$ as well. Finally, by the very flatness of the monoidal product the last property implies by \cref{tensorjointlymonic} that also the family of morphims 
$$\pi'_{Z_1}\ot\cdots \ot \pi'_{Z_n}:(\lim UF)^{\ot n}\to UFZ_1\ot\cdots\ot UFZ_n,$$ indexed by objects $Z_1,\ldots, Z_n$ in $\Zz$ is jointly monic. Combining the above and since $i:\lim F\to T\lim UF$ is a coalgebra map, we obtain that the family of morphisms
\[
\xymatrix{
\lim F  \ar[d]^-{i} \ar[r]^-{\Delta^n} & \lim F^{\ot n} \ar[rr]^-{\pi_{Z_1}\ot\cdots \ot \pi_{Z_n}} \ar[d]^-{i^{\ot n}}  && UFZ_1\ot\cdots\ot UFZ_n\\
T^c\lim UF \ar[r]^-{\Delta^n} &( T^c\lim UF)^{\ot n} \ar[rr]^-{p^{\ot n}}  && (\lim UF)^{\ot n} \ar[u]_-{\pi'_{Z_1}\ot\cdots \ot \pi'_{Z_n}} 
}
\]
indexed by $n\in \NN$ and objects $Z_1,\ldots,Z_n\in \Zz$, is jointly monic in $\Vv$.
\end{proof}

{
Let us finish this section with a few additional observations on cofree coalgebras, that will be useful for what follows
\begin{corollary}\label{corcofree}
Let $\Vv$ be a closed monoidal and locally presentable category, with very flat monoidal product.
\begin{enumerate}
\item If $\Vv$ is symmetric as monoidal category, then for any $V$ in $\Vv$, there exists an involutive isomorphism of coalgebras $T^c(V)\cong T^c(V)^{cop}$, which is natural in $V$.
\item The forgetful functor $U:\Coalg(V)\textsf{-}\Grph\to \Vv\textsf{-}\Grph$ has a right adjoint.
\end{enumerate}
\end{corollary}

\begin{proof}
\ul{(1)}. We know by definition that the projection $p:T^c(V)\to V$ is universal. Since the same $p$ can be viewed as a $\Vv$-morphism $T^c(V)^{cop}\to V$, we find by the universal property that there exists a unique coalgebra morphism $t:T^c(V)^{cop}\to T^c(V)$ such that $p\circ t=p$ in $\Vv$. Since then also $p\circ t\circ t = p$, the universal property of the cofree coalgebra implies that $t\circ t$ is the identity as required.\\
\ul{(2)}. One just applies \cref{Grphfunctor} to the forgetful-cofree adjunction between $\Vv$ and $\Coalg(\Vv)$.
\end{proof}
}

\section{Existence of free and cofree Hopf categories}

\subsection{Semi-Hopf $\Vv$-categories}

Recall from \cite{BCV} that a semi-Hopf category over a symmetric monoidal category $\Vv$ (where we denote the symmetry by $\sigma$) is nothing else than a $\Coalg(\Vv)$-category. Unwinding this definition, we obtain the following more explicit definition.

\begin{definition}
A {\em semi-Hopf $\Vv$-category} $A$ consists of a collection of objects $A^0$ and for all objects $x,y\in A^0$ we have a (coassociative, counital) coalgebra $(A_{xy},\delta_{xy}: A_{xy}\to A_{xy} \ot A_{xy},\epsilon_{xy}: A_{xy}\to I)$ in $\Vv$ together with the following morphisms in $\Vv$, for all $x,y,z\in A^0$:
\begin{eqnarray*}
m_{xyz}: A_{xy}\ot A_{yz} \to A_{xz} \qquad j_x: I \to A_{xx}
\end{eqnarray*}
which turn $A$ into a $\Vv$-category and such that moreover the following axioms hold:
\[
\xymatrix{A_{xy}\ot A_{yz} \ar[rr]^-{\delta_{xy}\ot \delta_{xy}} \ar[dd]_{m_{xyz}} & & A_{xy}\ot A_{xy}\ot A_{yz}\ot A_{yz} \ar[d]^{A_{xy}\ot \sigma \ot A_{yz}} \\
& & A_{xy}\ot A_{yz} \ot A_{xy}\ot A_{yz} \ar[d]^{m_ {xyz} \ot m_{xyz}} \\
A_{xz} \ar[rr]_{\delta_{xz}} & & A_{xz}\ot A_{xz} } 
\]
\[
\xymatrix{I \ar[r]^{\simeq} \ar[d]_{j_x} & I\ot I \ar[d]^{j_x\ot j_x} &A_{xy}\ot A_{yz} \ar[r]^-{\epsilon_{xy}\ot \epsilon_{xy}} \ar[d]_{m_{xyz}} & I\ot I \ar[d]^{\simeq} &I \ar[d]_{j_x} \ar[rd]^= & \\
A_{xx} \ar[r]_-{\delta_{xx}} & A_{xx}\ot A_{xx} &A_{xz} \ar[r]_{\epsilon_{xz}} & I &A_{xx} \ar[r]_{\epsilon_{xx}} & I } 
\]
A morphism between two semi-Hopf categories, called a semi-Hopf $\Vv$-functor is a morphism $\Coalg(\Vv)$-graphs that is at the same time a $\Vv$-functor.
Semi-Hopf $\Vv$-categories form a category which we denote by $\VsHopf$.
\end{definition}

\begin{examples}
Suppose that $\Vv=\Set$, or any Cartesian category viewed as a monoidal category, then $\Coalg(\Vv)=\Vv$, as any object $C$ of $\Vv$ can be endowed in a unique way with a comultiplication $C\to C\times C$ by means of the diagonal map. It then follows that a semi-Hopf $\Vv$-category is just a $\Vv$-category.

A semi-Hopf $\Vv$-category $\ul A$ (for any symmetric monoidal category $\Vv$) with $A^0$ containing just one element is nothing else than a Hopf algebra in $\Vv$. More generally, if $A^0$ is finite, then $\coprod_{x,y\in A^0} A_{xy}$ is a weak Hopf algebra in $\Vv$, see \cite[Proposition 6.1]{BCV} for a proof in the $k$-linear case.
\end{examples}

Combining the results from the previous sections, we then immediately arrive at the following.

\begin{proposition}\label{sHopflp}
Let $\Vv$ be a closed symmetric monoidal and locally presentable category. Then the following statements hold.
\begin{enumerate}[(i)]
\item The category $\VsHopf$ of semi-Hopf $\Vv$-categories is locally presentable.
\item The forgetful functor $\VsHopf\to \Coalg(\Vv)$-${\sf Grph}$ is monadic. In particular, the forgetful functor $\VsHopf\to\Coalg(\Vv)$-${\sf Grph}$ has a left adjoint, which creates the ``free semi-Hopf category'' over a $\Coalg(\Vv)$-graph.
\item The forgetful functor $\VsHopf\to \VCat$ is comonadic. In particular, the forgetful functor $\VsHopf\to\VCat$ has a right adjoint, which creates the ``cofree semi-Hopf category'' over a $\Vv$-category.
\end{enumerate}
\end{proposition}

\begin{proof}
{
\ul{(i)-(ii)}. Since $\Vv$ is a closed symmetric monoidal and locally presentable category, so is $\Coalg(\Vv)$ by \cref{coalglp}. Therefore, we can apply \cref{Vcatlp} to conclude that $\VsHopf=\Coalg(\Vv)-\Cat$ is locally presentable and the forgetful functor $\VsHopf\to \Coalg(\Vv)$-${\sf Grph}$ is monadic.}

\ul{(iii)}. The desired forgetful-cofree adjunction follows by combining \cref{coalglp} with \cref{Catfunctor}. The comonadicity follows directly from Becks precise tripleablity theorem, as  one can easily verify that the forgetful functor $\VsHopf\to \VCat$ preserves colimits and reflects isomorphisms. 
\end{proof}

\begin{remark}
The construction of the free semi-Hopf category over a $\Coalg(\Vv)$-graph $\ul A$ follows directly from the general construction as recalled in Remark~\ref{rem:freeVcat}(2). Moreover, as this construction is based on coproducts  of tensor products in $\Coalg(\Vv)$ and since the forgetful functor $\Coalg(\Vv)\to \Vv$  is monoidal and creates coproducts (see \cref{coalglp}), the whole construction can be performed simply in $\Vv$. In paricular, we find that the free semi-Hopf category has the same set of objects $A^0$ as the given $\Coalg(\Vv)$-graph $\ul A$, and for each $x,y\in A^0$, the Hom-object from $x$ to $y$ in the free semi-Hopf category is computed as the following coproduct in $\Vv$:
$$\coprod_{n\in \NN, z_1,\ldots, z_n \in A^0} A_{x,z_1}\ot A_{z_1,z_2}\ot \cdots \ot A_{z_n,y}$$
(with an additional component $I$ in case $x=y$), endowed with the natural comultiplication obtained from the comultiplications of the indvidual $A_{z,z'}$, and a composition induced by the associators of the monoidal product in $\Vv$.

The construction of the cofree semi-Hopf category over a given $\Vv$-category $\ul A$ can in a similar way be deduced from the construction of cofree coalgebras. Again, the set of objects will be given by $A^0$, the set of objects of the given $\Vv$-category $\ul A$. Then, one considers for any $x,y\in A^0$ the cofree coalgebra $T^c(A_{x,y})$ over the $\Vv$-object $A_{x,y}$ (which can be constructed explicitly in case the monoidal product is very flat following the construction from \cref{cofreeconstruction}). The composition (or multiplication) in the cofree semi-Hopf category is obtained from the universal property of the cofree coalgebra  (see paragraph after \cref{coalglp}) leading to the following commutative diagram.
\[\xymatrix{
T^c(A_{xy})\ot T^c(A_{yz}) \ar[rr]^-{p_{xy}\ot p_{yz}} \ar@{.>}[d] && A_{xy}\ot A_{yz} \ar[d]^-{m_{xyz}} \\
T^c(A_{xz}) \ar[rr]^-{p_{xz}} && A_{xz}
}\]
Finally, let us remark that semi-Hopf categories can also be viewed as coalgebra objects in a monoidal category of $\Vv$-categories with a fixed set of objects $X$. For any two $\Vv$-categories $\ul A=(X,A^1)$ and $\ul B=(X,B^1)$ over this set of objects, the monoidal product is given by $\ul A\bul \ul B=(X,A^1\bul B^1)$, where
$$(A\bul B)_{xy}=A_{xy}\ot B_{xy}.$$
See \cite{BCV} for this point of view on semi-Hopf categories and \cite{Bohm:polyads} for a bicategorical point of view on this which allows to vary the set of objects. Hence, the above description of the cofree semi-Hopf category over a $\Vv$-category can also be obtained directly from \cref{cofreeconstruction} applied to the monoidal category of $\Vv$-categories with a fixed set of objects.
\end{remark}

\begin{corollary}\label{familieslimcolim}
{ Let $\Vv$ be a symmetric closed monoidal and locally presentable category.}
Consider a small category $\Zz$ and a functor $F:\Zz\to\VsHopf$. 
\begin{enumerate}[(i)]
\item 
{ Consider the colimit $\colim F$ in $\VsHopf$ together with the canonical morphisms $\gamma_Z:FZ\to \colim F$.}
The set of objects of the colimit $\colim F$ is computed by taking the colimit of the functor $F$ composed with the forgetful functor to $\Set$.
For each pair of ojbects $x,y\in (\colim F)^0$ the following family of morphisms in $\Vv$
$$
\xymatrix{
(FZ_1)_{x_{11},x_{12}}\ot (FZ_2)_{x_{21},x_{22}} \ot \cdots \ot (FZ_n)_{x_{n1},x_{n2}} \ar[d]^-{(\gamma_{Z_1})_{x_{11}x_{12}}\ot (\gamma_{Z_2})_{x_{21}x_{22}}\ot\cdots \ot (\gamma_{Z_n})_{x_{n1}x_{n2}}} \\
(\colim F)_{xx_1} \ot (\colim F)_{x_1x_2} \ot \cdots \ot (\colim F)_{x_{n-1}y}  \ar[d]^-{m_{xx_1x_2\cdots x_{n-1}y}} \\ 
(\colim F)_{xy}
}
$$
varying over all $n\in \NN_0$, $Z_1,\ldots,Z_n\in \Zz$, $x_{i1},x_{i2}\in (UFZ_1)^0$ such that $\gamma^0_{Z_i}(x_{i2})=\gamma^0_{Z_{i+1}}(x_{i+1,1})=x_i$, $\gamma^0_{Z_1}(x_{11})=x$ and $\gamma^0_{Z_n}(x_{n2})=y$ is jointly epic (in $\Vv$). 
\item 
{ Consider the limit $\colim F$ in $\VsHopf$ together with the canonical morphisms $\pi_Z:\lim F\to FZ$.}
The set of objects of the limit $\lim F$ is computed by taking the limit of the functor $F$ composed with the forgetful functor to $\Set$.
If $\Vv$ has moreover a very flat monoidal product, then the family of morphisms
\[
\xymatrix{
(\lim F)_{x,y} \ar[rr]^-{\delta^n_{x,y}} && ((\lim F)_{x,y})^{\ot n} \ar[rrr]^-{\pi^1_{xy}\ot\cdots\ot \pi^n_{xy}} &&& (FZ_1)_{x_1,y_1} \ot\cdots\ot (FZ_n)_{x_n,y_n} 
}
\]
indexed by $n\in \NN$ and $Z_1,\ldots,Z_n$ in $\Zz$, is jointly monic in $\Vv$. Here for any tuple of objects $Z_1,\ldots,Z_n$ in $\Zz$, we denoted $\pi^0_{Z_i}(x)=x_i$ and $\pi^0_{Z_i}(y)=y_i$ and $(\pi_{x,y})_{Z_i}=\pi^i_{xy}$.
\end{enumerate}
\end{corollary}

\begin{proof}
{\ul (i)}. From \cref{sHopflp}(iii) we know that the colimit $\colim F$ can be computed in $\VCat$. The first part of the statement then follows from the fact (see \cref{Uvcattosetpreserves}) that the forgetful functor from $\VCat$ to $\Set$, being a left adjoint, preserves colimits. 
As explained in Remark \ref{rem:freeVcat}(3), we moreover find on each component of $\colim F$, the family of morphisms in $\Vv$ from the statement is jointly epic. 

{\ul (ii)}. From \cref{sHopflp}(ii) it follows that the limit $\lim F$ can be computed in $\Coalg(\Vv)$-${\sf Grph}$ and by \cref{Upreserves} we know that the forgetful functor from $\Coalg(\Vv)$-${\sf Grph}$ to $\Set$, being a right adjoint, preserves limits. This already shows the first part of the statement. For the second part, recall from \cref{limVGrph} how a limit in $\Coalg(\Vv)$-${\sf Grph}$ is build up from the limit of the underlying sets and limits in $\Coalg(\Vv)$. By \cref{limitsincoalg}, we know how limits in the category of $\Vv$-coalgebras can be constructed explicitly and that there is a jointly monic family of morphisms in $\Vv$ as in the statement.
\end{proof}

Recall that an op-monoidal functor $F:\Vv\to \Ww$ lifts to a functor $\Coalg(F):\Coalg(\Vv)\to \Coalg(\Ww)$. Moreover, if $\Vv$ and $\Ww$ are symmetric monoidal, and $F$ is symmetric op-monoidal, then $\Coalg(F)$ is a symmetric monoidal functor as well. Applying this in particular to the strong monoidal functor $(id,\sigma):\Vv\to \Vv^{rev}$, we obtain a strong monoidal functor $\Coalg(\Vv)\to \Coalg(\Vv)^{rev}$, since indeed $\Coalg(\Vv^{rev})=\Coalg(\Vv)^{rev}$ as monoidal categories. Moreover, using the symmetry once more, we also have a strong monoidal functor $\Coalg(\Vv)^{rev}\to \Coalg(\Vv)$. Combing these, we obtain a strong monoidal and involutive autofunctor $\Coalg(\Vv)\to \Coalg(\Vv)$, which acts as identity on morphisms.
The image of a coalgebra  $(C,\Delta,\epsilon)$ under this functor is called the {\em co-opposite} coalgebra of $C$ and is defined as the coalgebra $C^{cop}$ having the same underlying object and same counit, but whose comultiplication is given by 
$$\xymatrix{C \ar[rr]^-\Delta && C\ot C \ar[rr]^-\sigma && C\ot C}.$$

By applying the above to base-change from \cref{Catfunctor} in combination with \cref{oppositecat} and \cref{remoppositecat}, we immediately obtain the following.

\begin{proposition}\label{opcop}
Let $\Vv$ be a symmetric monoidal category, then we have the following involutive autofunctors
\begin{eqnarray*}
(-)^{op}&:&\VsHopf\to \VsHopf\\
(-)^{cop}&:&\VsHopf\to \VsHopf\\
(-)^{op,cop}&:&\VsHopf\to \VsHopf
\end{eqnarray*}
which send a semi-Hopf $\Vv$-category respectively to its opposite, (locally) coopposite, and opposite-coopposite semi-Hopf $\Vv$-category. The functors $(-)^{op}$ and $(-)^{op,cop}$ send morphisms to their opposite, and the functor $(-)^{cop}$ acts as identity on morphisms.
\end{proposition}

\subsection{Hopf $\Vv$-categories}\selabel{main}

Recall from \cite{BCV} the definition of a Hopf $\Vv$-category.

\begin{definition}
An {\em antipode} for a semi-Hopf $\Vv$-category $\ul A$ is a collection of $\Vv$-morphisms $S_{x,y}:A_{xy}\to A_{yx}$ satisfying the following axiom for all $x,y\in A^0$.
\[
\xymatrix{
& A_{x,y}\ot A_{x,y} \ar[r]^-{id\ot S_{x,y}} & A_{x,y}\ot A_{y,x} \ar[r]^-{m_{xyx}} & A_{xx}\\
A_{x,y} \ar[rr]^-{\epsilon_{x,y}} \ar[ur]^-{\delta_{xy}} \ar[dr]_-{\delta_{xy}}
&& I \ar[ur]_-{j_x} \ar[dr]^-{j_y}\\
& A_{x,y}\ot A_{x,y} \ar[r]^-{S_{x,y}\ot id} & A_{y,x}\ot A_{x,y} \ar[r]^-{m_{yxy}} & A_{yy}\\
}
\]
A semi-Hopf $\Vv$-category that admits an antipode is called a {\em Hopf $\Vv$-category}. We denote by $\VHopf$ the category of all Hopf $\Vv$-categories with semi-Hopf $\Vv$-functors between them.
\end{definition}

{ 
Recall from \cite[Proposition 3.10]{BCV} that any morphism of (semi-)Hopf categories automatically preserves the antipode. Consequently,  an antipode for a semi-Hopf category is unique whenever it exists. Moreover, by \cite[Theorem 3.7]{BCV}, the antipode defines an ``identity on objects'' morphism of semi-Hopf categories $S:\ul A\to \ul A^{op,cop}$. The latter means more precisely that $S:A\to A$ is an anti $\Vv$-category morphism (i.e.\ a contravariant $\Vv$-functor) and for each pair of objects $x,y\in A^0$, $S_{x,y}: A_{x,y}\to A_{y,x}$ is an anti-coalgebra morphism.
}

The following key-result of this paper shows that in favorable cases, limits and colimits of Hopf categories can be computed in the category of semi-Hopf categories.  
As a diagrammatic proof would be unreadable, with diagrams that don't fit on a page, we present the proof of this theorem by means of Sweedler notation on generalized elements. Let us explain this notation here. Consider a coalgebra $(C,\delta,\epsilon)$ in a symmetric monoidal category $\Vv$. Recall that a generalized element of $C$ is a pair $(A,a)$, where $A$ is any object of $\Vv$ and $a:A\to C$ is any morphism. We will then denote for any other map $f:C\to V$ the composition $f\circ a: A\to V$ suggestively as $f(a)$, expressing that (composition with) $f$ sends generalized elements of $C$ to generalized elements of $V$. Generalized Sweedler notation now involves writing for any $n\in \NN$, the composition $\delta^n\circ a=\delta^n(a):A\to C^{\ot n}$ as a formal tensor product:
$$\delta^n_{x,y}\circ a=a_{(1)}\ot\cdots\ot a_{(n)}:A\to C^{\ot n}.$$ 
If we now compose this map with a morphism of the form $f_1\ot \cdots \ot f_n:C^{\ot n}\to V_1\ot\cdots \ot V_n$, with $f_i:C\to V_i$ some morphisms in $\Vv$, we will denote the resulting map as
$$ f_1(a_{(1)})\ot \cdots \ot f_n(a_{(n)}) : A\to V_1\ot\cdots \ot V_n.$$
Furthermore, if we postcompose $a_{(1)}\ot\cdots\ot a_{(n)}$ with a combination of symmetry morphisms, we will denote this by simply permuting the corresponding indices in the Sweedler notation. For example, postcomposing $\delta\circ a=\delta(a)$ with the symmetry map $\beta_{C,C}:C\ot C\to C\ot C$ will be noted as
$$\beta_{C,C}\circ \delta\circ a = a_{(2)}\ot a_{(1)} : A\to C\ot C.$$
Since without loss of generality, we may assume that the family of morphisms $a$ is jointly epic in $\Vv$, to verify any equality in $\Vv$, it is sufficient to check it pre-composed with an arbitrary $a$ and use Sweedler notation to check the required equality on (generalized) elements, as we would do in the case of $\Vv=\Vect_k$.

\begin{proposition}\label{HopfinsHopf}
\begin{enumerate}
\item Let $\Vv$ be a symmetric monoidal and locally presentable category. Then $\VHopf$ is a cocomplete full subcategory of $\VsHopf$.
\item Let $\Vv$ be a symmetric monoidal, locally presentable category with very flat monoidal product. Then $\VHopf$ is a complete full subcategory of $\VsHopf$.
\end{enumerate}
\end{proposition}

\begin{proof}
By definition, $\VHopf$ is a full subcategory of $\VsHopf$. We should show that $\VHopf$ is closed in $\VsHopf$ under limits and colimits. 

\ul{(2)}.
Consider a small category $\Zz$ and a functor $F':\Zz\to \VHopf$. Let $F:\Zz\to\VsHopf$ be the composition of $F'$ with the inclusion functor $\VHopf\to\VsHopf$ and consider the limit $\lim F$ in $\VsHopf$. We will show that the semi-Hopf $\Vv$-category $\lim F$ has an antipode, hence $\lim F=\lim F'$.

Denote $\lim F=\ul L=(L^0,L^1)$ and denote for any object $Z\in \Zz$ the canonical projections of the limit by $\pi_Z:L\to FZ$.
Since the functor $(-)^{op,cop}:\VsHopf\to\VsHopf$ is an isomorphism of categories (see \cref{opcop}), it commutes with limits. Hence, denoting $(-)^{op,cop}\circ F=F^{op,cop}$, we find that $\lim F^{op,cop}=(\lim F)^{op,cop}$, and the projections $\pi_Z^{op}:\lim F^{op,cop}\to F_Z$ are the opposite of the projections $\pi_Z$ (i.e.\ $(\pi^{op}_Z)_{xy}=(\pi_Z)_{yx}$).
{ The antipodes $S_{FZ}$ of the Hopf categories $FZ$ induce a cone} $S_{FZ}\circ \pi_Z:\lim F\to FZ^{op,cop}$ on $F^{op,cop}$ (in the category of semi-Hopf categories). Therefore we obtain a unique morphism of semi-Hopf categories $S:\lim F\to \lim F^{op,cop}=(\lim F)^{op,cop}$ such that $\pi_Z^{op}\circ S=S_{FZ}\circ \pi_Z$. We claim that $S$ is an antipode for $\lim F$. 

Now fix $x,y\in L^0$.
By \cref{familieslimcolim}\footnote{Remark that at this point we use the very flat monoidal product.}, we know that (using the same notation as in \cref{familieslimcolim}) the family of morphisms $(\pi^1_{yy}\ot\cdots\ot \pi^n_{yy})\circ \delta^n_{yy}$ indexed over all $n\in \NN$ and $Z_1,\ldots,Z_n\in\Zz$ is jointly monic in $\Vv$. Hence, to check the antipode property $m_{yxy}\circ (S_{xy}\ot id)\circ \delta_{xy}$ it is enough to check it post-composed with these morphisms.  As explained just above the proposition, we check this equality on an arbitrary generalized element $a:A\to L_{xy}$.
Then we find
\begin{eqnarray*}
&&(\pi^1_{yy}\ot\pi^2_{yy}\ot\cdots\ot \pi^n_{yy})\circ \delta^n_{yy}\circ m_{yxy}\circ (S_{xy}\ot id)\circ \delta_{xy}(a)\\
&=&(\pi^1_{yy}\ot\pi^2_{yy}\ot\cdots\ot \pi^n_{yy})\circ \delta^n_{yy}\circ m_{yxy}(S_{xy}(a_{(1)})\ot a_{(2)}))\\
&=& \pi^1_{yy}\circ m_{yxy}\big(S_{xy}(a_{(n)})\ot a_{(n+1)}\big)\ot \pi^2_{yy}\circ m_{yxy}\big(S_{xy}(a_{(n-1)})\ot a_{(n+2)})\big)\ot \cdots \\
&& \hspace{8cm}\cdots \ot \pi^n_{yy}\circ m_{yxy}\big(S_{xy}(a_1)\ot a_{(2n)})\big)\\
&=& m_{y_1x_1y_1}^1\big(S^1_{x_1y_1}(\pi^1_{xy}(a_{(n)})_{(1)})\ot \pi^1_{xy}(a_{n})_{(2)}\big)\ot \pi^2_{yy}\circ m_{yxy}\big(S(a_{(n-1)})\ot a_{(n+1)}\big)\ot \cdots\\
&& \hspace{8cm}\cdots \ot \pi^n_{yy}\circ m_{yxy}\big(S(a_1)\ot a_{(2n-1)}\big)\\
&=& j^1_{y_1y_1}\epsilon^1_{x_1y_1}\big(\pi^1_{xy}(a_{(n)})\big) \ot \pi^2_{yy}\circ m_{yxy}\big(S(a_{(n-1)})\ot a_{(n+1)}\big)\ot \cdots\\
&& \hspace{8cm}\cdots \ot \pi^n_{yy}\circ m_{yxy}\big(S(a_1)\ot a_{(2n-1)}\big)\\
&=& \pi^1\circ j_{y} \epsilon_{xy}(a_{(n)}) \ot \pi^2_{yy}\circ m_{yxy}\big(S(a_{(n-1)})\ot a_{(n+1)}\big)\ot \cdots\\
&& \hspace{8cm}\cdots \ot \pi^n_{yy}\circ m_{yxy}\big(S(a_1)\ot a_{(2n-1)}\big)\\
&=& \pi^1_{yy}\circ j_{y} \ot \pi^2_{yy}\circ m_{yxy}\big(S(a_{(n-1)})\ot a_{(n)}\big)\ot \cdots\ot \pi^n_{yy}\circ m_{yxy}\big(S(a_{(1)})\ot a_{(2n-2)}\big)\\ 
&=& \cdots\\
&=& \pi^1_{yy}\circ j_{y} \ot \pi^2_{yy}\circ j_{y} \ot \cdots \ot \pi^n_{yy}\circ j_{y} \epsilon_{xy}(a)\\
&=& (\pi^1\ot \cdots \ot \pi^n)\circ \delta^n_{yy} \circ j_{y} \epsilon_{xy}(a)
\end{eqnarray*}
Here we used in the second equality the compatibility between multiplication (composition) and comultiplication in $L$, together with coassociativity and the fact that the antipode morphism $S_{xy}$ is an anticoalgebra morphism. The third equality follows from the fact that $\pi^1$ is a semi-Hopf category morphism, combined with the construction of $S$ which tells us that $\pi^1_{yx}\circ S_{xy}=S^1_{xy}\circ \pi^1_{xy}$. The forth equality is the antipode property for the Hopf category $FZ_1$. The fifth equality follows from the fact that $\pi^1$ is a morphism of semi-Hopf categories, hence it commutes with the units $j$ and counits $\epsilon$. The sixth equality follows from the counit property of the coalgebra $L_{xy}$. Then we use an induction step. The last equality follows from the compatibility between comultiplication and unit in the semi-Hopf category $L$.

Since $(\pi^1_{yy}\ot\cdots\ot \pi^n_{yy})\circ \delta^n_{yy}$ is a jointly monic family and $a:A\to L_{xy}$ was an arbitrary generalized element, we can conclude from the above calculation that $m_{yxy}\circ (S_{xy}\ot id)\circ \delta_{xy}=j_{y} \epsilon_{xy}$ for all $x,y \in L^0$. Similarly, one shows that $L$ satisfies the right antipode property, hence $L$ is a Hopf category.

\ul{(1)}. 
Consider again a functor $F':\Zz\to\VHopf$ and compose it with the forgetful functor to semi-Hopf categories to obtain the functor $F:\Zz\to\VsHopf$. As the proof of part \ul{(2)}, we consider $\colim F$ in $\VsHopf$ and build a candidate antipode on this object.  \cref{familieslimcolim} provides a jointly epic family in $\Vv$ on each Hom-object of $\colim F$\footnote{In this case, the monoidal product is not required to be very flat.}. By a similar computation as in part \ul{(1)}, one then proves the antipode property by pre-composing it with this jointly epic family. 
\end{proof}

{
We are now ready to state the main theorem of this paper and provide a very short proof, based on the previous theorem showing that the forgetful functor from Hopf categories to semi-Hopf categories preserves all limits and colimits, and the earlier results on local presentability. 
}

\begin{theorem}\label{freecofreeHopfcat}
Let $\Vv$ be a closed symmetric monoidal and locally presentable category with very flat monoidal product. Then $\VHopf$ is locally presentable and the fully faithful inclusion functor $\VHopf\to\VsHopf$ has a left adjoint $H$ and right adjoint $H^c$. That is, for any Hopf category $\ul H$ and any semi-Hopf category $\ul A$, we have that
\begin{eqnarray}
\VsHopf(\ul A,\ul H) &=& \VHopf( H\ul A,\ul H) \label{freeadjunction}\\
\VsHopf(\ul H,\ul A) &=& \VHopf( \ul H,H^c\ul A) \label{cofreeadjunction}
\end{eqnarray}
\end{theorem}

\begin{proof}
By \cref{sHopflp}, the category $\VsHopf$ is locally presentable and since $\VHopf$ is a complete and cocomplete full subcategory of $\VsHopf$ by \cref{HopfinsHopf}, it follows from \cref{sublp} that $\VHopf$ is itself locally presentable. Moreover, as the inclusion functor $\VHopf\to \VsHopf$ preserves all limits an colimits, it has both a left and right adjoint by \cref{adjointlp}.
\end{proof}

\section{On the construction of free and cofree Hopf categories}\selabel{construction}

In the previous section, we proved the existence of free and co-free Hopf categories. The aim of this section is to sketch explicitly the construction of the free Hopf $\Vv$ category $H\ul A$ and the cofree Hopf category $H^c\ul A$ over a given semi-Hopf $\Vv$-category $\ul A$,  where $\Vv$ is a symmetric closed monoidal finitely presentable category with very flat monoidal product. As we already observed before, the free and cofree Hopf categories over $\ul A$ has the same set of objects as $\ul A$, so $HA^0=A^0=H^cA^0$. The Hom-objects of the free and cofree Hopf categories are described in the following two subsections.

\subsection{The free Hopf category}

Let us now describe the construction of Hom-objects for the free Hopf category, $HA^1$. This construction goes in 3 steps.

\ul{Step 1}. Consider any $x,y\in A^0$. We will transform the coalgebra $A_{xy}$ into a larger $\Vv$-coalgebra, to have enough room to freely contain the image of the antipode. To this end, we define for any $i\in \NN$, a coalgebra
$$A^{(i)}_{xy}:=\left\{\begin{array}{rcl} 
A_{xy} & {\text{if}} & 2 \mid i\\
A_{yx}^{cop} & {\text{if}} & 2\nmid i
\end{array}\right.$$
Then we consider the coalgebra $A'_{xy}:= \coprod_{i\in \NN} A^{(i)}_{xy}$ in $\Vv$. In this way, we obtain the $\Coalg(\Vv)$-$\Grph$ $\ul A'=(A^0,A'_{xy})$. Moreover, there is an identity-on-objects morphism of $\Coalg(\Vv)$-graphs $\ul \iota:\ul A\to \ul A'$, defined for all $x,y\in A^0$ by the inclusion $A_{xy} = A^{(0)}_{xy} \to A'_{xy}$. Furthermore, $A'$ is naturally equipped with a $\Coalg(\Vv)$-$\Grph$ morphism $s:A'\to A'^{op,cop}$, induced by the coalgebra morphisms $s^{(i)}_{xy}:A^{(i)}_{xy}\to (A^{(i+1)}_{yx})^{cop}$, which are all given by identities. Moreover, for each even $i\in \NN$ and each odd $j\in \NN$ we have a morphism of $\Coalg(\Vv)$-graphs 
\begin{eqnarray*}
\iota^{(i)}:\ul A\to \ul A',&& \iota^{(i)}_{xy}:A_{xy}\to A^{(i)}_{xy}\\
\iota^{(j)}:\ul A^{op,cop}\to \ul A',&& \iota^{(j)}_{xy}:A_{yx}^{cop}\to A^{(j)}_{xy}
\end{eqnarray*}

\ul{Step 2}. Now we consider the free semi-Hopf category over the $\Coalg(\Vv)$-$\Grph$ $\ul A'$. That is, we apply the left adjoint $T$ to the forgetful functor $\textsf{$\Coalg(\Vv)$-$\Cat$}\to \textsf{$\Coalg(\Vv)$-$\Grph$}$ from \cref{Vcatlp} and as described in Remark \ref{rem:freeVcat}(2). Explicitly, we obtain at this stage for any $x,y\in (T\ul A')^0=A^0$
$$(TA')_{xy}= \coprod_{n\in \NN, z_1,\ldots,z_n\in A^0} A'_{xz_1}\ot \cdots \ot A'_{z_ny},$$
adding an additional component of the monoidal unit $I$ to this coproduct if $x=y$.

We can then endow $T\ul A'$ with a morphism of semi-Hopf $\Vv$-categories $\ul S': T\ul A'\to (T\ul A')^{op,cop}$. In each component, $\ul S'$ is given by an anticoalgebra morphism $S'_{xy}: (TA')_{xy}\to (TA')_{yx}$ obtained by reversing the ``path'' from $x$ to $y$ and applying $s_{wz}:A'_{wz} \to A'_{zw}$ for all pairs $(w,z)$ of consecutive indices in the path. Since $\Vv$ is a symmetric monoidal category, we can realize this morphism by a composition a suitable symmetry isomorphism with a tensor product of $s_{wz}$:
$$
\xymatrix{A'_{x,x_1}
\ot \cdots \ot A'_{x_n,y} \cong A'_{x_n,y} \ot \cdots 
\ot A'_{x,x_1} 
\ar[rrr]^-{s_{x_n,y}\ot \cdots 
\ot s_{x,x_1}} 
&&& A'_{y,x_n}\ot \cdots 
\ot A'_{x_1,x}}$$
And furthermore $S'_{xx}$ acts as the identity on the additional component $I$. Composing the morphisms $\iota^{(i)}$ defined in Step 1 with the canonical morphism of $\Coalg(\Vv)$-graphs $\ul A'\to T\ul A'$ we obtain morphism of $\Coalg(\Vv)$-graphs, for $i$ even and $j$ odd:
\begin{eqnarray*}
\iota^{(i)'}:\ul A\to T\ul A',&& \iota^{(i)'}_{xy}:A_{xy}\to TA^{(i)}_{xy}\\
\iota^{(j)'}:\ul A^{op,cop}\to T\ul A',&& \iota^{(j)'}_{xy}:A_{yx}^{cop}\to TA^{(j)}_{xy}.
\end{eqnarray*}

\ul{Step 3}.
We ``factor out'' certain relations by means of a suitable coequalizer $q:T\ul A'\to H\ul A$ in $\Vv$-$\Grph$, such that $H\ul A$ becomes the universal Hopf category we are looking for. In other words, we want $q:T\ul A'\to H\ul A$ to be the ``largest quotient\footnote{By ``quotient'' here we mean the underlying morphism of $q$ in $\Vv$-Grph is a (regular) epimorphism, as it was already stated before. This regular epimorphism can be described explicitly, and means intuitively that one has to quotient out not only the relations itself, but the ``ideal'' generated by them.}'' in the category of $\Vv$-categories such that the following conditions hold. 
\begin{enumerate}[(a)]
\item The morphisms $q\circ \iota^{(i)'}$ are morphisms of $\Vv$-categories. (Quotienting this relation on the free $\Vv$-categories turns the coproduct of copies of $A$ and $A^{op,cop}$ into a coproduct in the category of $\Vv$-categories).
\item The antipode conditions (which do not necessarily hold in $T\ul A'$), become valid in $H\ul A$, that is, $q_{xx}:(T\ul A')_{xx}\to (H\ul A)_{xx}$ coequalizes all pairs $(m_{xyx}\circ (id \ot S'_{xy})\circ \delta_{xy}, j_x\circ \epsilon_{xy})$ and $m_{xyx}\circ (S'_{yx}\ot id)\circ \delta_{yx}, j_x\circ \epsilon_{yx})$ for each $x,y\in X$.
\end{enumerate}
Then one can show that $H\ul A$ naturally also has a unique $\Coalg(\Vv)$-graph structure (this requires a non-trivial but technical verification, see e.g. \cite[Crucial Lemma]{Porst:formal1} for the one-object case) such that it becomes a semi-Hopf category and $q\circ \iota^{(i)'}$ are morphisms of semi-Hopf categories. Moreover one obtains a morphism $\ul S:H\ul A\to (H\ul A)^{op,cop}$ such that $\ul S\circ q=q^{op}\circ \ul S'$ turning $H\ul A$ into a Hopf $\Vv$-category.
Although it is quite clear which relations are to be factored out, it is quite involved to verify in detail that every step is justified and this construction gives indeed the free Hopf category over $\ul A$. For some additional details, we refer to \cite{Gross}.

From the construction above, one sees that for any $x,y\in A^0$ we have the following jointly epic family of morphisms in $\Vv$
$$\xymatrix{A_{(xz_1)}\ot \cdots \ot A_{(z_ny)} \ar[rrr]^-{\iota^{(i_1)}_{xz_1}\ot \cdots \ot \iota^{(i_{n+1})}_{z_ny}} &&& TA_{xy} \ar[rr]^-{q_{xy}} && HA_{qx,qy}}$$
indexed by $n, i_1,\ldots, i_{n+1}\in \NN$ and $z_1,\ldots,z_n\in A^0$. We have put the indices in the first tensor product between brackets, to indicate that their order needs to be reversed whenever the corresponding $i_k$ is odd.

A given morphism of semi-Hopf categories $f:\ul A\to \ul H$, where $\ul H$ is a Hopf category, then induces a unique morphism of Hopf categories $\tilde f:H\ul A\to \ul H$ under the bijection \eqref{freeadjunction}. By the above jointly epic family, $\tilde f$ is completely determined by means of its pre-composition with the morphisms in this family. In this way, $\tilde f$ is defined by the following morphisms
$$\xymatrix{
A_{(xz_1)}\ot \cdots \ot A_{(z_ny)} \ar[rr]^-{f\ot \cdots \ot f} && H_{(xz_1)}\ot \cdots \ot H_{(z_ny)} \ar[rr]^{S_{xz_1}^{i_1} \ot \ldots \ot S_{z_ny}^{i_{n+1}}} && H_{xz_1}\ot \cdots \ot H_{z_ny} \ar[r]^-m &
H_{xy}
}$$ 

\subsection{The cofree Hopf category}

As in the case of the free Hopf category, we construct the Hom-objects of the cofree Hopf category in three steps.

\ul{Step 1}. For any $x,y\in A^0$, we set for any $i\in \NN$
$$A^{(i)}_{xy}:=\left\{\begin{array}{rcl} 
A_{xy} & {\text{if}} & 2 \mid i\\
A_{yx} & {\text{if}} & 2\nmid i
\end{array}\right.$$
Then define $A'_{xy}=\prod_{i\in \NN} A^{(i)}_{xy}$, which defines a new $\Vv$-graph $\ul A'$ over $A^0$. Moreover this $\Vv$-graph is a $\Vv$-category by means of the compositions $m'_{xyz}:A'_{xy}\ot A'_{yz}\to A'_{xz}$ and units $j'_{xx}:I\to A'_{xx}$ defined by the commutativity of the following diagrams
\[
\xymatrix{
A'_{xy}\ot A'_{yz} \ar[rr]^-{m'_{xyz}} \ar[d]_-{\pi^{(i)}_{xy}\ot \pi^{(i)}_{yz}} && A'_{xz} \ar[d]^-{\pi^{(i)}_{xz}} \\
A^{(i)}_{xy}\ot A^{(i)}_{yz} \ar[rr]^-{m^{(i)}_{xyz}} && A^{(i)}_{xz}
}\qquad
\xymatrix{
I \ar[rrd]_-{j_{xx}} \ar[rr]^-{j'_{xx}} && A'_{xx} \ar[d]^-{\pi^{(i)}_{xx}} \\
&& A^{(i)}_{xx}=A_{xx}}
\]
In the above diagram, the morphism $\pi$ denote the obvious projection morphisms from the product in $\Vv$, $m^{(i)}_{xyz}$ is just the composition $m_{xyz}$ of $\ul A$ if $i$ is even and it is the opposite composition $m_{zyx}\circ \sigma$ in case $i$ is odd.
Moreover consider for each $i\in \NN$ and $x,y\in A^0$ the morphism $s^{(i)}_{xy}:A_{xy}^{(i)}\to A_{yx}^{(i+1)}$ given by identity. Then these induce a morphism of $\Vv$-categories $s:\ul A'\to \ul A'^{op}$. 
Moreover, for each even $i\in \NN$ and each odd $j\in \NN$ we have a morphism of $\Vv$-categories
\begin{eqnarray*}
\pi^{(i)}:\ul A'\to \ul A,&& \pi^{(i)}_{xy}: A^{(i)}_{xy}\to A_{xy}\\
\pi^{(j)}:\ul A'\to \ul A^{op},&& \pi^{(j)}_{xy}:A^{(j)}_{xy} \to A_{yx}
\end{eqnarray*}
which are defined as the identity in each component.

\ul{Step 2}. Now we take the cofree semi-Hopf $\Vv$-category over the $\Vv$-category $\ul A'$ as seen in \cref{sHopflp} and subsequent remark. In other words, we replace each Hom-object $A'_{xy}$ by the cofree coalgebra over it $T^cA'_{xy}$. 
{ Applying the cofree functor to the morphism of $\Vv$-categories $s$, we obtain a morphism of semi-Hopf categories $T^c(s):T^c\ul A' \to T^c(\ul A^{op}) = (T^c\ul A)^{op}$. Moreover, we know from \cref{corcofree}(1) that for each pair of objects $x,y$ there is a natural coalgebra isomorphism $T^cA'_{xy} \cong (T^cA_{xy})^{cop}$. Combining these with $T^c(s)$, we obtain a} morphism of semi-Hopf categories $S':T^c\ul A'\to (T^c\ul A')^{op,cop}$\footnote{Remark that the isomorphism $T^cA'_{xy} \cong (T^cA_{xy})^{cop}$ can be build up from the symmetry of the tensor product, which highlights again the similarity with the free construction explained above.}. Precomposing the morphisms $\pi^{(i)}$ from Step 1 with the cofree projection morphism $p:T^c\ul A'\to \ul A$, we obtain the following morphisms for all $i\in \NN$ even and $j\in \NN$ odd
\begin{eqnarray*}
\pi^{(i)'}:T^c\ul A'\to \ul A,&& \pi^{(i)'}_{xy}: T^cA^{(i)}_{xy}\to A_{xy}\\
\pi^{(j)'}:T^c\ul A'\to \ul A^{op},&& \pi^{(j)'}_{xy}:T^cA^{(j)}_{xy} \to A_{yx}
\end{eqnarray*}

\ul{Step 3}. Finally, following an approach similar to the constructions in \cref{cofreeconstruction} and \cref{limitsincoalg}, one constructs the largest $Coalg(\Vv)$-subgraph $\iota : H^c\ul A\to T^c\ul A'$, such that the following conditions hold.
\begin{enumerate}[(a)]
\item Each of the compositions $\pi^{(i)'}\circ \iota$ is a morphism of $\Coalg(\Vv)$-graphs. 
\item The antipode conditions (which do not necessarily hold in $T^c\ul A'$), become valid in $H\ul A$. This means that for all $x,y$ in $X$, the morphisms $\iota_{xy}$ equalizes the pairs  $(m_{xyx}\circ (id \ot S'_{xy})\circ \delta_{xy}, j_x\circ \epsilon_{xy})$ and $m_{yxy}\circ (S'_{xy}\ot id)\circ \delta_{xy}, j_y\circ \epsilon_{xy})$.
\end{enumerate}
Then $H^c$ can be endowed naturally with a $\Vv$-category structure (again, this a non-trivial point for which we refer to \cite[Crucial Lemma]{Porst:formal1} for the one-object case), turning it into a semi-Hopf $\Vv$-category such that $\pi^{(i)'}\circ \iota$ is a morphism of semi-Hopf $\Vv$-categories. Moreover, we obtain a morphism of semi-Hopf categories $S:H^c\ul A\to (H^c\ul A)^{op,cop}$ such that $\iota^{op}\circ S=S'\circ \iota$, and this turns $H\ul A$ in a Hopf category.
Again, verifying that the above construction indeed yields the cofree Hopf category over $\ul A$ is rather involved. 

From the above construction, and the construction of the cofree coalgebra as given in \cref{cofreeconstruction}, we see that for any pair $x,y\in A^0$, we have the following jointly monic family of morphisms in $\Vv$.
$$\xymatrix{H^cA_{xy} \ar[rrr]^-{\pi^{(i_1)}_{xy}\ot \cdots \ot \pi^{(i_{n+1})}_{xy}} &&&
A_{(xy)}\ot \cdots \ot A_{(xy)} }$$
indexed by $n, i_1,\ldots, i_{n+1}\in \NN$ and $z_1,\ldots,z_n\in A^0$. Again, we have put the indices in latter tensor product between brackets, to indicate that their order needs to be reversed whenever the corresponding $i_k$ is odd.

Given a morphism of semi-Hopf categories $f:\ul H\to \ul A$, where $\ul H$ is a Hopf category then induces a unique morphism of Hopf categories $\tilde f:\ul H\to H^c\ul A$ under the bijection \eqref{cofreeadjunction}. By the above jointly monic family, $\tilde f$ is completely determined by means of its composition with the morphisms in this family. In this way, $\tilde f$ is defined by the following morphisms
$$\xymatrix{
H_{xy} \ar[r]^-{\delta_{xy}^n} & 
H_{xy}\ot \cdots \ot H_{xy}
\ar[rr]^{S_{xy}^{i_1} \ot \ldots \ot S_{xy}^{i_{n+1}}} && H_{(xy)}\ot \cdots \ot H_{(xy)}
 \ar[rr]^-{f\ot \cdots \ot f} && A_{(xy)}\ot \cdots \ot A_{(xy)}  
}$$

\section{Conclusions and outlook}

The results of this paper can now be summarized in the following diagram, which can be constructed for any symmetric closed monoidal and locally presentable category $\Vv$ with very flat monoidal product. The arrows going from left to right in this diagram are `forgetful' functors, the arrows from right to left are the free and cofree constructions, so left and right adjoints (as indicated by the adjunction symbol), which we constructed in this paper. The lower part of the diagram is the classical one-object situation which we mentioned in the introduction. The diagram shows how we indeed lifted this to the multi-object setting, since the vertical arrows indicate the functor that send internal structures in $\Vv$ to the one-object versions above. As the diagram commutes, our multi-object construction indeed strictly generalizes the one-object constructions.
\begin{equation}\label{diagramconclusion}
\xymatrix{
\VHopf \ar[drrr] \\
&&& \VsHopf \ar@<2ex>[ulll]^-{\ref{freecofreeHopfcat}}  \ar@<-2ex>[ulll]_-{\ref{freecofreeHopfcat}} \ar@<1ex>@{}[ulll]|-{\nwvdashhh} \ar@<-1ex>@{}[ulll]|-{\nwvdashhh}
  \ar@<1ex>[rr]_{\top} \ar@<1ex>[dr] && \Coalg(\Vv){\textsf{-Grph}} \ar@<1ex>[dr] \ar@<1ex>[ll]^-{\ref{sHopflp}(ii)} \\
\Hopf(\Vv) \ar[drrr] \ar[uu]&&& & \VCat \ar@<1ex>[rr]_-{\top\hspace{1cm}} \ar@{}[ul]^(.25){}="a"^(.95){}="b" \ar@<1ex> "a";"b" ^-{\ref{sHopflp}(iii)} \ar@{}[ul]|-{\nwvdashh}  && \VGrph \ar@<1ex>[ll]^-{\ref{Vcatlp}\qquad} \ar@{}[ul]^(.25){}="a"^(.95){}="b" \ar@<1ex> "a";"b" ^-{\ref{corcofree}} \ar@{}[ul]|-{\nwvdashh} \\
&&& \Bialg(\Vv) \ar@<2ex>[ulll]  \ar@<-2ex>[ulll] \ar@<1ex>@{}[ulll]|-{\nwvdashhh} \ar@<-1ex>@{}[ulll]|-{\nwvdashhh}
\ar@<1ex>[rr]|(0.45){\phantom{x}}_{\hspace{.5cm}\top} \ar[uu] \ar@<1ex>[dr] && \Coalg(\Vv) \ar[uu]|(0.44){\phantom{x}}|(0.56){\phantom{x}} \ar@<1ex>[dr]  \ar@<1ex>[ll]|(0.55){\phantom{x}} \\
&&& & \Alg(\Vv) \ar@<1ex>[rr]_{\top} \ar[uu] \ar@{}[ul]^(.25){}="a"^(.95){}="b" \ar@<1ex> "a";"b" \ar@{}[ul]|-{\nwvdashh} && \Vv \ar[uu] \ar@<1ex>[ll] 
\ar@{}[ul]^(.25){}="a"^(.95){}="b" \ar@<1ex> "a";"b" ^-{\ref{cofreeconstruction}}
\ar@{}[ul]|-{\nwvdashh} 
} 
\end{equation}
Moreover, all functors in the upper part of the above diagram commute with the forgetful functors to $\Set$. Hence, if we consider $\Vv$-categories, semi-Hopf $\Vv$-categories or Hopf $\Vv$-categories with a fixed set of objects, then the free and cofree constructions restrict and corestrict to these subcategories. This holds in particular for the one-object case as we remarked just above. 
 
When we consider the particular instance $\Vv=\Set$, we recover another well-known case. Firstly, observe that since $\Set$ is Cartesian, $\Coalg(\Set)=\Set$. Hence $\Set$-categories and semi-Hopf $\Set$-categories coincide, and are just small categories. Hopf $\Set$-categories are groupoids. The right adjoint of the forgetful functor ${\textsf{$\Set${\sf-sHopf}}}\to{\textsf{$\Set${\sf-Cat}}}$ associates to category $A$ the groupoid of all isomorphims in $A$. The left adjoint is more involved and constructs the free groupoid over a category, see e.g.\ \cite{Brown}.

Another case of particular interest, is given by the semi-Hopf category of algebras, which are known to be enriched over coalgebras by means of Sweedler's universal measuring coalgebras \cite{Sweedler, Tambara, AGV}. Our construction of the (co)free Hopf category over a semi-Hopf category allow to turn the category of algebras into a Hopf category. This construction is then similar to what is done in \cite{Manin} to obtain general quantum symmetry groups of non-commutative spaces. However, an important remark should be made here. In the Hopf category of algebras which we would obtain as explained above, the endo-Hom-object of an algebra $A$ would naturally act on $A$. In \cite{Manin} however, a dual situation is considered, where the Hopf category would ``coact'' on $A$. However, in order to obtain such a universal coacting object, a finiteness condition on $A$ is required. Indeed, in \cite{Manin}, $A$ is supposed to be a graded algebra that is finite dimensional in each degree. This finiteness condition has been refined in \cite{AGV}. In fact, following the constructions of \cite{Manin} and \cite{AGV}, one would obtain a Hopf opcategory, rather then a Hopf category. Recall that in a Hopf $\Vv$-opcategory $A$, the Hom-objects $A_{xy}$ are algebras in $\Vv$, which are endowed with ``cocomposition morphisms'' $\Delta_{xyz}:A_{xz}\to A_{xy}\ot A_{yz}$ and counits $\epsilon_x:A_{xx}\to I$, satisfying dual conditions as those of a Hopf-category. The natural question arises whether the free and cofree construction of the present paper can also be obtained in the setting of Hopf-opcategories. Let us remark that in \cite{Vasi}, that co-categories rather then opcategories are considered, and these are shown to be comonadic over $\Vv$-graph and inheriting locally presentability from $\Vv$. Whether opcategories and Hopf-opcategories also are locally presentable remains a question for future investigations and should be studied along with (Sweedler) duality between Hopf-categories and Hopf-opcategories.

\subsection*{Acknowledgements}
The authors would like to thank warmly Christina Vasilakopoulou, for interesting discussions and useful explanations of locally presentable categories. We also would like to thank warmly the referee for the careful reading of our paper, and the many suggestions and corrections which greatly improved the manuscript.

For this research, JV would like to thank the FWB (f\'ed\'eration Wallonie-Bruxelles) for support through the ARC project “From algebra to combinatorics, and back”. PG thanks the FNRS (Fonds de la Recherche Scientifique) for support through a FRIA fellowship.


\begin{thebibliography}{99}
\bibitem{AdaRos}
J. Ad\'amek, J.\ Rosick\'y, Locally presentable and accessible categories, London Mathematical Society Lecture Note Series, 189, Cambridge University Press, Cambridge, 1994. xiv+316 pp. 

\bibitem{Ago1}
A.\ Agore, Categorical constructions for Hopf algebras, {\em Comm. Algebra} {\bf 39} (2011), 1476--1481. 

\bibitem{Ago2}
A.\ Agore, Limits of coalgebras, bialgebras and Hopf algebras, {\em proceedings of the american mathematical society}, {\bf 139} (2011), 855--863.

\bibitem{AGV}
A.\ Agore, A.\ Gordienko and J.\ Vercruysse, $V$-universal Hopf algebras (co)acting on $\Omega$-algebras, 
{\em Commun. Contemp. Math.} {\bf 25} (2023), Paper No. 2150095, 40 pp.

\bibitem{AGM}
A.\ Ardizzoni, I.\ Goyvaerts, C.\ Menini, Pre-Rigid Monoidal Categories, {\em  Quaest. Math.} (2022), to appear. (doi:10.2989/16073606.2022.2112106)

\bibitem{BCV}
E.\ Batista, S.\ Caenepeel and J.\ Vercruysse, Hopf categories, {\em Algebr. Represent. Theory} {\bf 19} (2016), 1173--1216.

\bibitem{BCSW}
R.\ Betti, A.\ Carboni, R.\ Street, R.\ Walters, Variation through enrichment, {\em J. Pure Appl. Algebra} {\bf 29} (1983), 109--127.

\bibitem{BloLer}
R.\ Block, P.\ Leroux
Generalized dual coalgebras of algebras, with applications to cofree coalgebras
{\em J. Pure Appl. Algebra} {\bf 36} (1985), 15--21.

\bibitem{Bohm:Hoids}
G.\ B\"ohm, Hopf algebroids. Handbook of algebra. Vol. 6, 173--235, Handb. Algebr., 6, Elsevier/North-Holland, Amsterdam, 2009.

\bibitem{Bohm:polyads}
G.\ B\"ohm, Hopf polyads, Hopf categories and Hopf group monoids viewed as Hopf monads, {\em Theory Appl. Categ.} {\bf 32} (2017), Paper No. 37, 1229--1257.

\bibitem{BNS:WHA}
G.\ B\"ohm, F.\ Nill, K.\ Szlach\'anyi, Weak Hopf algebras. I. Integral theory and $C^*$-structure, {\em J. Algebra} {\bf 221} (1999), 385--438. 


\bibitem{Brown}
R.\ Brown, Topology and Groupoids, BookSurge, Charleston, SC, 2006. 


\bibitem{BFVV}
M.\ Buckley, T.\ Fieremans, C.\ Vasilakopoulou, and J.\ Vercruysse, A Larson-Sweedler Theorem for Hopf $\Vv-$categories, {\em Advances in Mathematics}, {\bf 376}  (2021), Paper No. 107456, 64 pp.

\bibitem{BFVV2}
M.\ Buckley, T.\ Fieremans, C.\ Vasilakopoulou, and J.\ Vercruysse,  Oplax Hopf algebras, {\em High. Struct.} {\bf 5} (2021), 71--120.


\bibitem{CaeFie}
S.\ Caenepeel, T.\ Fieremans, Descent and Galois theory for Hopf categories, {\em J. Algebra Appl.} {\bf 17} (2018), 1850120, 39 pp.

\bibitem{Fox}
T.F.\ Fox,
The construction of cofree coalgebras,
{\em J. Pure Appl. Algebra},
{\bf 84}
(1993), 191--198.

\bibitem{GV}
I.\ Goyvaerts, J.\ Vercruysse, On the duality of generalized Lie and Hopf algebras, {\em Adv. Math.} {\bf 258} (2014), 154--190.

\bibitem{Gross}
P.\ Grosskopf, Hopf categories, Frobenius categories and Homotopy Quantum Field Theories, PhD thesis, Université Libre de Bruxelles, 2024.

\bibitem{Chris}
M.\ Hyland, I.\ L\'opez Franco, C.\ Vasilakopoulou, Hopf measuring comonoids and enrichment, {\em Proc. Lond. Math. Soc.} (3) {\bf 115} (2017), 1118--1148. 

\bibitem{JanKel}
G.\ Janelidze, G.M.\ Kelly, 
A note on actions of a monoidal category
{\em Theory Appl. Categ.} {\bf 9} (2001), 61--91.

\bibitem{Kelly}
G.M.\ Kelly, Basic Concepts of Enriched Category Theory, Cambridge University Press, 1982.

\bibitem{KelLac}
G.M.\ Kelly, S.\ Lack, $\Vv$-Cat is locally presentable or locally bounded if $\Vv$ is so, {\em Theory Appl. Categ.} {\bf 8} (2001), 555--575.


\bibitem{Manin}
Yu.\  Manin, I. Quantum groups and noncommutative geometry. Université de Montreal, Centre de Recherches Mathematiques, Montreal, QC, 1988. 



\bibitem{Porst}
H.-E.\ Porst, On categories of monoids, comonoids, and bimonoids, {\em Quaest. Math.} {\bf 31} (2008), 127--139.

\bibitem{Porst:JPAA}
H.-E.\ Porst, Universal constructions for Hopf algebras, {\em J. Pure Appl. Algebra} {\bf 212} (2008), 2547-–2554.

\bibitem{Porst:formal1}
H.-E.\ Porst, The formal theory of Hopf algebras, Part I,  {\em Quaest. Math.} {\bf 38} (2015), 631--682.

\bibitem{Porst:formal2}
H.-E.\ Porst, The formal theory of Hopf algebras, Part II, {\em Quaest. Math.} {\bf 38} (2015), 683--708.

\bibitem{Ros}
J.\ Rosick\'y, Enriched purity and presentability in Banach spaces, arXiv:2206.08546.

\bibitem{Sweedler} 
M.\ E.\ Sweedler, Hopf Algebras. W. A. Benjamin New York, 1969.

\bibitem{Takeuchi} 
M\ Takeuchi, Free Hopf algebras generated by coalgebras.
{\em J. Math. Soc. Japan}, \textbf{23} (1971), 561--582.

\bibitem{Tambara}
D.\ Tambara, The coendomorphism bialgebra of an algebra. {\em J. Fac. Sci. Univ. Tokyo Math.}, \textbf{37} (1990), 425--456.


\bibitem{Vasi}
C.\ Vasilakopoulou, Enriched Duality in couble categories: $\Vv$-categories and $\Vv$-cocategories, {\em Journal of Pure and Applied Algebra} {\bf 223} (7),  July 2019, 2889--2947.

\bibitem{Ver:local}
J.\ Vercruysse, Local units versus local projectivity. Dualisations : Corings with local structure maps, {\em Communications in Algebra} {\bf 34} (2006), 2079--2103. 


\bibitem{Wolff}
H.\ Wolff,
$V$-cat and $V$-graph,
{\em J. Pure Appl. Algebra} {\bf 4} (1974), 123--135.

\bibitem{ZH}
B.\ Zimmermann-Huisgen, Pure submodules of direct products of free modules, {\em Math.
Ann.} {\bf 224} (1976), 233--245.

\end{thebibliography}
\end{document}